\documentclass[12pt]{article}
\usepackage[letterpaper]{geometry}

\usepackage{authblk}
\usepackage{graphicx,graphics,epsfig,epsf,listings,url,subfigure}
\usepackage{amsmath,amsthm,amssymb,amsfonts,mathrsfs}
\usepackage{bbm,bm}
\usepackage{enumitem}
\usepackage{multirow}
\usepackage{color}

\newcommand{\vect}[1]{\left[ \begin{array}{c} #1 \end{array} \right]}
\newcommand{\subcom}[2]{_{\{#1\}#2}}
\usepackage[figuresright]{rotating}

\newtheorem{theorem}{Theorem}

\newtheorem{remark}{Remark}
\newtheorem{definition}{Definition}

\begin{document}

\title{Partitioned and implicit-explicit general linear methods for ordinary differential equations\footnote{This paper is dedicated to Prof. J.C. Butcher's 80-th birthday.}}

\author[1]{Hong Zhang}
\author[1]{Adrian Sandu\thanks{Corresponding author. Tel. 540-231-2193.}}
\affil{Computational Science Laboratory \\ Department of Computer Science\\ Virginia Polytechnic Institute and State University\\ Blacksburg, VA 24061\\ E-mails: \{zhang,sandu\}@vt.edu.}

\date{\today}

\maketitle

\begin{abstract}
Implicit-explicit (IMEX) time stepping methods can efficiently solve differential equations with both stiff and nonstiff components. 
IMEX Runge-Kutta methods and IMEX linear multistep methods have been studied in the literature.
In this paper we study new implicit-explicit methods of general linear type (IMEX-GLMs). We develop an order conditions theory for high stage order partitioned GLMs
that share the same abscissae, and show that no additional coupling order conditions are needed. Consequently, GLMs offer an excellent framework for the construction of
multi-method integration algorithms. Next, we propose a family of IMEX schemes based on 
diagonally-implicit multi-stage integration methods and construct practical schemes of order three. 
Numerical results confirm the theoretical findings.

\end{abstract}

\newpage
\thispagestyle{empty}
\tableofcontents
\newpage
\setcounter{page}{2}

\section{Introduction}
Implicit-explicit (IMEX) time integration schemes are becoming increasingly popular for solving multiphysics problems with both stiff and nonstiff components, 
which arise in many application areas such as mechanical and chemical engineering, astrophysics, meteorology and oceanography, and environmental science.
Examples of multiphysics problems with both stiff and nonstiff components include advection-diffusion-reaction equations, 
fluid-structure interactions, and Navier-Stokes equations. 
Such problems can be expressed concisely as the system of ordinary differential equations (ODEs)
\begin{equation}
 \label{eqn:ode_imex}
y' = f(t,y) + g(t,y)\,
\quad t_0 \le t \le t_F\,, \quad y(t_0) = y_0\,, 
\end{equation}
where $f$ corresponds to the nonstiff term, and $g$ corresponds to the stiff term. 
In case of systems of partial differential equations (PDEs) the system \eqref{eqn:ode_imex}
appears after semi-discretization in space.

An IMEX scheme treats the nonstiff term explicitly and the stiff term implicitly, therefore 
combining the low cost of explicit methods with the favorable stability properties of implicit methods. 
IMEX linear multistep methods have been developed in \cite{Ascher_1995,Frank_1997_IMEXstability,Hundsdorfer_2007_IMEX_LMM}, 
and IMEX Runge-Kutta methods have been built in 
\cite{Ascher_1997,Boscarino_2009_IMEXuniform,Pareschi_2000,Verwer_2004}.

The general linear method (GLM) family proposed by J.C Butcher \cite{Butcher_1993a} generalizes both Runge-Kutta and linear multistep methods.
The added complexity improves the flexibility to develop methods with better stability and accuracy properties. 
While Runge-Kutta and linear multistep methods are special cases of GLMs, the framework allows for the construction of many other methods as well. 
Here we focus on the diagonally implicit multistage integration methods (DIMSIM) \cite{Butcher_1993b}, which are both efficient and accurate, and 
 great potentials for practical use. 
GLM can overcome the limitations of both linear multistep methods (lack of A-stability at high orders) and of Runge-Kutta methods (low stage order f
leading to order reduction). A complete treatment of GLMs can be found in the book of Jackiewicz \cite{Jackiewicz_2009_book}.

In this study we develop the concept of partitioned DIMSIM methods, and develop an order conditions theory
for a family of such methods. This shows that partitioned GLM is a great framework for developing multi-methods.
Next, we propose a new family of implicit-explicit methods based on pairs of DIMSIMs, and develop second and third order methods on this class.

In our earlier work \cite{Zhang_2012} we have developed second order IMEX-GLM schemes.
While this paper was under study, we became aware of an effort to construct IMEX-GLM schemes for Hamiltonian systems \cite{Butcher_2013_imex-glm}.

The paper is organized as follows. Section \ref{sec:GLM} reviews the class of general linear methods.
The new concept of partitioned DIMSIM schemes is proposed in Section \ref{sec:partitioned-GLM}, and the order conditions theory is developed.
IMEX-DIMSIM schemes are constructed in Section \ref{sec:IMEX-schemes}. Linear stability is analizes in section \ref{sec:PR-convergence},
and Prothero-Robinson convergence in section \ref{sec:PR-convergence}. IMEX methods of second and third order are built in Sections 
\ref{sec:IMEX-order2} and \ref{sec:IMEX-order3a}, respectively.
Numerical results for van der Pol system and for the two dimensional gravity waves equations are presented in Section \ref{sec:results}.
Section \ref{sec:conclusions} draws conclusions and points to future work.

\section{General linear methods}\label{sec:GLM}

\subsection{Representation of general linear methods}\label{sec:GLM-representation}
Consider the initial value problem for an autonomous system of differential equations in the form
\begin{equation}
\label{eqn:ivp}
 \displaystyle y'(t) =f(y)\,, \quad t_0 \le t \le t_F\,, \quad y(t_0) = y_0 \,,
\end{equation}
with $f: \mathbb{R}^d \rightarrow \mathbb{R}^d$ and $y(t) \in \mathbb{R}^d$. 
GLMs \cite{Jackiewicz_2009_book} for (\ref{eqn:ivp}) can be represented by the abscissa vector $\mathbf{c} \in \mathbb{R}^{s}$, and four coefficient matrices
$\mathbf{A} \in \mathbb{R}^{s \times s}$, $\mathbf{U} \in \mathbb{R}^{s \times r}$, $\mathbf{B} \in \mathbb{R}^{r \times s}$ and $\mathbf{V} \in \mathbb{R}^{r \times r}$ which can be represented compactly in the following tableau 
\begin{center}
  \begin{tabular}{c|c}
   $\mathbf{A}$ & $\mathbf{U}$ \\ \hline
   $\mathbf{B}$ & $\mathbf{V}$ \\
\end{tabular}\,.
\end{center}
On the uniform grid $t_n=t_0 + nh$, $n=0,1,\dots,N$, $Nh=t_F-t_0$, one step of the GLM reads
\begin{subequations}
\label{eqn:GLMs} 
\begin{eqnarray}
 Y_i &=& h\sum_{j=1}^s a_{i,j} \, f(Y_j)+\sum_{j=1}^r u_{i,j} \, y_j^{[n-1]}, \: i=1,\dots,s,~ \\
 y_i^{[n]}&=& h\sum_{j=1}^s b_{i,j} \, f(Y_j) + \sum_{j=1}^r v_{i,j} \, y_j^{[n-1]}, \: i=1,\dots,r,~
\end{eqnarray}
\end{subequations}
where $s$ is the number of internal stages and r is the number of external stages. 
Here, $h$ is the step size, $Y_i$ is an approximation to $y(t_{n-1}+c_i h)$ and $y_i^{[n]}$ is an approximation to the linear combination of the derivatives of $y$ at the point $t_n$.
The method \eqref{eqn:GLMs}  can be represented in vector form 
\begin{subequations}
\label{eqn:GLMs_vector} 
\begin{eqnarray}
  Y &=& h \, \left(\mathbf{A} \otimes \mathbf{I}_{d \times d} \right) \, F(Y)+ \left(\mathbf{U} \otimes \mathbf{I}_{d \times d} \right)\, y^{[n-1]},~ \\
 y^{[n]} &=& h \, \left( \mathbf{B} \otimes \mathbf{I}_{d \times d} \right) \, F(Y) + \left(\mathbf{V} \otimes  \mathbf{I}_{d \times d} \right)\, y^{[n-1]},
\end{eqnarray}
\end{subequations}
where $\mathbf{I}_{d \times d} $ is an identity matrix of the dimension of the ODE system.

\subsection{Stability considerations}\label{sec:stability}
The linear stability of method \eqref{eqn:GLMs} is analyzed in terms of its stability matrix
\begin{equation}
\label{eqn:GLM_stability_matrixr} 
 \mathbf{M}(z) = \mathbf{V} + z \, \mathbf{B} \, \left(\mathbf{I}_{s \times s} -z\mathbf{A} \right)^{-1}\, \mathbf{U}\,,
\end{equation}
and the corresponding stability function
\begin{equation}
\label{eqn:GLM_stability_function} 
 p(w,z) = \det(w \mathbf{I}_{r \times r} - \mathbf{M}(z)),
\end{equation}
where $w,z \in \mathbb{C}$.
A desirable property is the inherited Runge-Kutta stability \cite{Wright_2002,Butcher_2003}.
This means that the stability function \eqref{eqn:GLM_stability_function}  has the form
\begin{equation}
 p(w,z) = w^{s-1} \, \bigl(w-R(z)\bigr)\,,
\end{equation}
where $R(z)$ is the stability function of Runge Kutta method of order $p=s$. 

\subsection{Accuracy considerations}\label{sec:accuracy}

We assume that the components of the input vector $y_i^{[n-1]}$ for the next step in \eqref{eqn:GLMs} satisfy 
\begin{equation}
\label{eqn:GLM-order_assumption}
 y_i^{[n-1]} = \sum_{k=0}^p q_{i,k} h^k \, y^{(k)} (t_{n-1}) + \mathcal{O}(h^{p+1}), \quad i = 1,\dots,r,
\end{equation}
for some real parameters $q_{i,k}$, $i =1,\dots,r$, $k=0,1,\dots,p$. 

The method \eqref{eqn:GLMs} has {\it order} $p$ if the output vector $y_i^{[n]}$ satisfies
\begin{equation}
\label{eqn:GLM-order}
 y_i^{[n]} = \sum_{k=0}^p q_{i,k} h^k y^{(k)} (t_n) + \mathcal{O}(h^{p+1}), \quad i = 1,\dots,r,
\end{equation}
for the same parameters $q_{i,k}$ of \eqref{eqn:GLM-order_assumption}.

The method \eqref{eqn:GLMs} has {\it stage order} $q$ if the internal stage vectors $Y_i^{[n]}$ are approximations of order $q$ to the solution at the time points $t_{n-1}+c_i \, h$
\begin{equation}
\label{eqn:GLM-stage-order}
 Y_i^{[n]} = y(t_{n-1}+c_i \, h) + \mathcal{O}(h^{q+1}), \quad i = 1,\dots,s\,.
\end{equation}
We collect the parameters $q_{i,k}$ in the matrix $\mathbf{W}$ for convenience
\begin{equation}
 \mathbf{W}=[\mathbf{q}_0 \quad \mathbf{q}_1 \quad \cdots \quad \mathbf{q}_p] = \left[
\begin{tabular}{c c c c}
$q_{1,0}$ & $q_{1,1}$ & $\cdots$ & $q_{1,p}$ \\
 $q_{2,0}$ & $q_{2,1}$ & $\cdots$ & $q_{2,p}$ \\
$\vdots$ & $\vdots$ & $\ddots$ & $\vdots$ \\
$q_{r,0}$ & $q_{r,1}$ & $\cdots$ & $q_{r,p}$ \\
\end{tabular}
\right]. 
\end{equation}
\begin{theorem}[GLM order conditions \cite{Jackiewicz_2009_book}]
 Assume that $y^{[n-1]}$ satisfies (\ref{eqn:GLM-order_assumption}). Then the GLM (\ref{eqn:GLMs}) has order $p$ (\ref{eqn:GLM-order}) and stage order $q=p$ 
(\ref{eqn:GLM-stage-order}) if and only if
\begin{subequations}
\label{eqn:GLM-order_conditions}
\begin{eqnarray}
\label{eqn:GLM-order_condition-1}
e^{cz} = z\mathbf{A}e^{cz} + \mathbf{U}w(z) + \mathcal{O}(z^{p+1}), \\
\label{eqn:GLM-order_condition-2}
e^z w(z) = z \mathbf{B} e^{cz} + \mathbf{V} w(z) + \mathcal{O}(z^{p+1})\,,
\end{eqnarray}
where 
\[
e^{cz} = [e^{c_1z}, \dots, e^{c_sz}]^T\,, \quad w(z)= \sum_{j=0}^p \mathbf{q}_j \, z^j\,.
\]
For stage order $q=p-1$ condition \eqref{eqn:GLM-order_condition-1} is replaced by
\begin{eqnarray}
\label{eqn:GLM-order_condition-1-b}
e^{cz} &=& z\mathbf{A}e^{cz} + \mathbf{U}w(z) + \left(\frac{\mathbf{c}^p}{p!} - \frac{\mathbf{A}\, \mathbf{c} }{(p-1)!} - \mathbf{U} \mathbf{q}_p\right)z^p + \mathcal{O}(z^{p+1})\,.
\end{eqnarray}
\end{subequations}
\end{theorem}
\begin{proof}
See Butcher \cite{Butcher_1993a} and Jackiewicz \cite[Section 2.4]{Jackiewicz_2009_book}.
\end{proof}

It is shown in \cite{Jackiewicz_2009_book} that a GLM \eqref{eqn:GLMs} has order $p$ and stage order $q$ with $q = p = r = s$ 
if and only if
\begin{equation}
\label{eqn:coeB}
\mathbf{B} = \mathbf{B}_0 -\mathbf{A} \mathbf{B}_1 - \mathbf{V} \mathbf{B}_2 + \mathbf{V} \mathbf{A},
\end{equation}
where the matrices $\mathbf{B}_{0},\mathbf{B}_{1},\mathbf{B}_{2} \in \mathbb{R}^{s \times s}$ are defined by
\begin{equation}
\left( \mathbf{B}_0 \right)_{i,j} =  \frac{\int _0^{1+c_i} \phi _j(x) d x}{\phi_j(c_j)}\,, \quad  
\left( \mathbf{B}_1 \right)_{i,j} =  \frac{\phi_j(1+c_i)}{\phi_j(c_j)}\,, \quad   
\left( \mathbf{B}_2 \right)_{i,j} = \frac{\int_0 ^{c_i} \phi_j(x) d x}{\phi_j(c_j)}\,, 
\end{equation}
with
\begin{displaymath}
 \phi_i(x) = \prod_{j=1,j \neq i}^s (x-c_j), \quad  i=1,\dots,s\,.
\end{displaymath}
%

\subsection{Starting and ending procedures}\label{sec:starting}
Assumption \eqref{eqn:GLM-order_assumption} requires to compute the initial vector $y^{[0]}$ by a starting procedure satisfying
\begin{equation}
 y_i^{[0]} = \sum_{k=0}^p q_{i,k} h^k y^{(k)} (t_0) + \mathcal{O}(h^{p+1}), \quad i = 1,\dots,r\,.
\end{equation}
Dense output is based on derivative approximations of the form
\begin{equation}
\label{eqn:dense_output}
h^k y^{(k)} (t_n) \approx \sum_{i=0}^s \beta_{k,i} f(Y_i) + \sum_{j=0}^r \gamma_{k,j} y_j^{[n-1]}, \quad k=0,1,\dots,r\,.
\end{equation}
It is shown in \cite{Butcher_1993a,Butcher_1997} that \eqref{eqn:dense_output} is accurate within $\mathcal{O}(h^{p+1})$ if and only if
\begin{equation}
\label{eqn:BV}
[1,z,\dots,z^p]^T e^z = z \widetilde{\mathbf{B}} e^{cz} + \widetilde{\mathbf{V}} w(z) + \mathcal{O}(z^{p+1})
\end{equation}
where $\widetilde{\mathbf{B}}= [\beta_{k,i}]$ and $\widetilde{\mathbf{V}}=[\gamma_{k,i}]$. 
The termination procedure uses \eqref{eqn:dense_output} with $k=0$ to generate the solution at the last step
\begin{equation}
\label{eqn:solution}
y(t_n) \approx \sum_{i=0}^s \beta_{0,i} f(Y_i) + \sum_{j=0}^r \gamma_{0,j} y_j^{[n-1]}.
\end{equation}
%

\subsection{Diagonally implicit multistage integration methods}\label{sec:DIMSIM}
Diagonally implicit multistage integration methods (DIMSIMs) are a subclass of GLMs characterized by the following properties \cite{Butcher_1993a}:
\begin{enumerate}
 \item $\mathbf{A}$ is lower triangular with the same element $a_{i,i}=\lambda$ on the diagonal;
\item $\mathbf{V}$ is a rank-1 matrix with the nonzero eigenvalue equal to one to guarantee preconsistency;
\item The order $p$, stage order $q$, number of external stages $r$, and number of internal stages $s$ are related by $q\in \{ p-1, p\}$ and $r\in \{ s,s+1\}$.
\end{enumerate}
In this work we focus on DIMSIMs with $p=q=r=s$, $\mathbf{U}= \mathbf{I}_{s \times s}$, and $\mathbf{V}=\mathbf{1}_s\, v^T$, where $v^T\, \mathbf{1}_s=1$ \cite{Jackiewicz_2009_book}.
DIMSIMs can be categorized into four types according to \cite{Butcher_1993a}. 
Type 1 or type 2 methods have $a_{i,j} = 0$  for $j \ge i$ and are suitable for a sequential computing environment, while type 2 and type 3 methods 
have $a_{i,j} = 0$  for $j \ne i$ and are suitable for parallel computation. Methods of type 1 and 3 are explicit ($a_{i,i} = 0$), while methods of type 2
and 4 are implicit ($a_{i,i} = \lambda \neq 0$) and potentially useful for stiff systems.

%
%
%

\section{Partitioned general linear methods}\label{sec:partitioned-GLM}

Consider the partitioned system of ODEs
\begin{equation}
\label{eqn:partitioned-ode}
 y'= \vect{ y_{\{1\}}\\ \vdots\\ y_{\{N\}} }' = \vect{ f_{\{1\}} (y_{\{1\}},\dots, y_{\{N\}}) \\ \vdots \\ f_{\{N\}}( y_{\{1\}},\dots, y_{\{N\}} ) } = f(y)\,,
\end{equation}
where the solution vector is separated into components $y_{\{m\}}$, $m=1,\dots,N$, each of which may be itself a vector.

A {\it partitioned general linear method} solves \eqref{eqn:partitioned-ode} by applying a different GLM to each component.
If not explicitly stated otherwise , we use the subscript $\{m\}$ to denote the coefficients specific to the $m$-th component of the partitioned system.
We have the following

\begin{definition}[Partitioned GLM]
One step of a partitioned GLM has the form
\begin{subequations}
 \label{eqn:partitoned_glm}
\begin{eqnarray}
 \label{eqn:partitoned_glm_stage}
  Y_{\{m\}i} &=& h\sum_{j=1}^s a_{\{m\}i,j} \, f_{\{m\}} (Y_{\{1\}j},Y_{\{2\}j},\dots, Y_{\{N\}j}) \\
\nonumber  && + \sum_{j=1}^r u_{\{m\}i,j} \, y_{\{m\}j} ^{[n-1]}, \: i=1,\dots,s, \\
 \label{eqn:partitoned_glm_final}
 y_{\{m\}i}^{[n]} &=& h\sum_{j=1}^s b_{\{m\}i,j} \, f_{\{m\}} (Y_{\{1\}j},Y_{\{2\}j},\dots, Y_{\{N\}j} ) \\
\nonumber  && + \sum_{j=1}^r v_{\{m\}i,j} \, y_{\{m\}j}^{[n-1]}, \: i=1,\dots,r,
\end{eqnarray}
\end{subequations}
where $a_{\{m\}i,j}, u_{\{m\}i,j} , b_{\{m\}i,j}$, and $c\subcom{m}{i}$ for $m=1,\dots,N$ represent the coefficients of $N$ different GLMs. 
\end{definition}

\begin{definition}[Internal consistency]
A partitioned GLM \eqref{eqn:partitoned_glm} is {\it internally consistent} if all component methods share the same abscissae, $\mathbf{c}_{\{m\}i}=\mathbf{c}_i$ for $m=1,\dots,N$.
\end{definition}
Internal consistency means that all stage components approximate the solution components at the same time point, i.e., $[Y_{\{1\}j}, ,\dots, Y_{\{N\}j}] \approx y(t_n+c_j h)$, for all $j=1,\dots,s$.
An internally consistent partitioned GLM method \eqref{eqn:partitoned_glm} can be represented compactly as
\begin{equation}
\label{eqn:compact-partitioned-glm}
\mathbf{c}_{\{m\}}=\mathbf{c}\,,\qquad
  \begin{array}{c|c}
   \mathbf{A}_{\{m\}} & \mathbf{U}_{\{m\}} \\[2pt] \hline 
   \mathbf{B}_{\{m\}} & \vspace*{2pt} \mathbf{V}_{\{m\}} 
\end{array}  \,.
\end{equation}

\begin{definition}[Order of partitioned GLM]
Assume that each component of the input vector satisfies \eqref{eqn:GLM-order_assumption}
\begin{equation}
\label{eqn:GLM-order_assumption_partitioned}
 y_{\{m\}\,i}^{[n-1]} = \sum_{k=0}^p q_{\{m\} i,k} h^k \, y_{\{m\}}^{(k)} (t_{n-1}) + \mathcal{O}(h^{p+1}), \quad i = 1,\dots,r\,.
\end{equation}
The partitioned GLM \eqref{eqn:partitoned_glm} has order $p$ if each component of the output vector satisfies
\begin{equation}
\label{eqn:GLM-order_partitioned}
 y_{\{m\}\,i}^{[n]} = \sum_{k=0}^p q_{\{m\}\,i,k} h^k y_{\{m\}}^{(k)} (t_n) + \mathcal{O}(h^{p+1}), \quad i = 1,\dots,r \,, \quad m=1,\dots,N\,,
\end{equation}
for the same parameters $q_{\{m\}i,k}$ as in \eqref{eqn:GLM-order_assumption_partitioned}.
The partitioned GLM  \eqref{eqn:partitoned_glm} has stage order $q$ if each component of the internal stages $Y_i^{[n]}$ satisfies
\begin{equation}
\label{eqn:GLM-stage-order_partitioned}
 Y_{\{m\}\,i}^{[n]} = y_{\{m\}}(t_{n-1}+c\subcom{m}{i} \, h) + \mathcal{O}(h^{q+1}), \quad i = 1,\dots,s\,, \quad m=1,\dots,N \,.
\end{equation}
\end{definition}

\begin{theorem}[Order conditions for partitioned GLMs]
Assume that each component $y_{\{m\}j}^{[n-1]}$ satisfies (\ref{eqn:GLM-order_assumption}). Then the internally consistent partitioned GLM \eqref{eqn:compact-partitioned-glm} 
has order  $p$ \eqref{eqn:GLM-order_partitioned} and
\and stage order $q \in \{p-1,p\}$ \eqref{eqn:GLM-stage-order_partitioned}
if and only if each component method $\left(\mathbf{A}_{\{m\}},\mathbf{B}_{\{m\}},\mathbf{U}_{\{m\}},\mathbf{V}_{\{m\}}\right)$ has order $p$  \eqref{eqn:GLM-order} and stage order $q$ \eqref{eqn:GLM-stage-order}.
\end{theorem}

\begin{remark}
Each component method needs to independently meet its own order conditions \eqref{eqn:GLM-order_conditions}.  No additional ``coupling'' conditions are needed for the partitioned GLM (i.e.,
no order conditions contain coefficients from multiple component schemes).
\end{remark} 

\begin{proof}

We first prove the ``only if'' part: if the partitioned GLM satisfies  \eqref{eqn:GLM-order_partitioned}--\eqref{eqn:GLM-stage-order_partitioned} with order $p$ \and stage order $q \in \{p-1,p\}$, then each component method satisfies its own order conditions \eqref{eqn:GLM-order}--\eqref{eqn:GLM-stage-order} with the same $p$ and $q$. This can be seen immediately by employing the same component method for all partitions, $\left(\mathbf{A}_{\{k\}},\mathbf{B}_{\{k\}},\mathbf{U}_{\{k\}},\mathbf{V}_{\{k\}}\right)$ $\equiv$ $\left(\mathbf{A}_{\{m\}},\mathbf{B}_{\{m\}},\mathbf{U}_{\{m\}},\mathbf{V}_{\{m\}}\right)$ for $k=1,\dots,N$.
The partitioned method  \eqref{eqn:compact-partitioned-glm}  is the traditional GLM method $\left(\mathbf{A}_{\{m\}}\right.$, $\mathbf{B}_{\{m\}}$, $\mathbf{U}_{\{m\}}$, $\left. \mathbf{V}_{\{m\}}\right)$ and has to satisfy the traditional order conditions  \eqref{eqn:GLM-order} and \eqref{eqn:GLM-stage-order}.

We next prove the ``if'' part: if each component method satisfies  \eqref{eqn:GLM-order}--\eqref{eqn:GLM-stage-order} with order $p$ \and stage order $q \in \{p-1,p\}$, then the partitioned GLM \eqref{eqn:compact-partitioned-glm}  has order $p$ and stage order $q$. Denote
\[
Y_j = \begin{bmatrix} Y\subcom{1}{j} \\ \vdots \\ Y\subcom{N}{j} \end{bmatrix}\,, \quad
Y = \begin{bmatrix} Y_{1}\\ \vdots \\ Y_{s} \end{bmatrix}\,,
\]
and
\[
y(t_{n-1}+c_{j} h) = \begin{bmatrix}  y\subcom{1}{}(t_{n-1}+c_j h)\\ \vdots \\ y\subcom{N}{}(t_{n-1}+c_j h) \end{bmatrix}\,, \quad
y(t_{n-1}+\mathbf{c} h) = \begin{bmatrix} y(t_{n-1}+c_{1} h)\\ \vdots\\ y(t_{n-1}+c_{s} h) \end{bmatrix}\,.
\]

Consider the stage equations of the individual method $m$ with exact solution arguments
\begin{eqnarray}
\label{eqn:stage-with-exact-solution-full}
y(t_{n-1}+c_{i} h) &=& h\sum_{j=1}^s a_{\{m\}i,j} \, f \left(y(t_{n-1}+c_j h) \right) \\
\nonumber  && + \sum_{j=1}^r u_{\{m\}i,j} \, \left( \sum_{k=0}^p q_{\{m\}\,i,k} h^k y^{(k)} (t_{n-1})  \right) + \mathcal{O}\left(h^{q+1}\right) , ~~ i=1,\dots,s\,. 
\end{eqnarray}
The error size is given by the stage order $q$ of each individual method \eqref{eqn:GLM-stage-order}. 
Using the assumption \eqref{eqn:GLM-order_assumption_partitioned} each component of the sum $\sum_{k=0}^p q_{\{m\}\,i,k} h^k y_{\{m\}}^{(k)} (t_{n-1})$
can be replaced by the numerical approximations $y_{\{m\}j} ^{[n-1]}$, which differ from their exact counterparts by
$\mathcal{O}(h^{p+1})$; therefore their use in \eqref{eqn:stage-with-exact-solution-full} does not change the asymptotical error size.
The $m$-th component of relation \eqref{eqn:stage-with-exact-solution-full} then reads
\begin{eqnarray}
\label{eqn:stage-with-exact-solution}
y_{\{m\}}(t_{n-1}+c_{i} h) &=& h\sum_{j=1}^s a_{\{m\}i,j} \, f_{\{m\}} \left(y(t_{n-1}+c_j h) \right) \\
\nonumber  && + \sum_{j=1}^r u_{\{m\}i,j} \,  y_{\{m\}j}^{[n-1]} + \mathcal{O}\left(h^{q+1}\right) , ~~ i=1,\dots,s\,. 
\end{eqnarray}
Subtracting \eqref{eqn:stage-with-exact-solution} from the stage equation \eqref{eqn:partitoned_glm_stage} gives
\begin{eqnarray*}
&& Y\subcom{m}{i}-y_{\{m\}}(t_{n-1}+c_{i} h) = h\sum_{j=1}^s a_{\{m\}i,j} \, 
\left( f_{\{m\}} (Y_j ) - f_{\{m\}} (y(t_{n-1}+c_{j} h)) \right) + \mathcal{O}(h^{q+1}) 
\end{eqnarray*}
and therefore
\begin{eqnarray*}
\left\| Y\subcom{m}{i}-y_{\{m\}}(t_{n-1}+c_{i} h) \right\|_\infty 
&\le& h \left\| \mathbf{A}_{\{m\}} \right\|_\infty \, L_m\, \left\| Y-y(t_{n-1}+\mathbf{c} h) \right\|_\infty + \mathcal{O}(h^{q+1}) 
\end{eqnarray*}
where $L_m$ is the Lipschitz constant of $f_{\{m\}}$. It follows that \cite{Jackiewicz_2009_book}
\begin{equation}
\label{eqn:stage-error-bound}
\left\| Y-y(t_{n-1}+\mathbf{c} h) \right\|_\infty = \mathcal{O}(h^{q+1})
\end{equation}
for all sufficiently small step sizes 
\[
h < \tau =  \left( \max_m \left\| \mathbf{A}_{\{m\}} \right\|_\infty \, L_m \right)^{-1}\,.
\]
Equation \eqref{eqn:stage-error-bound} proves the stage order of the partitioned GLM method.

Continuing,  \eqref{eqn:stage-error-bound} implies that
\begin{eqnarray}
\nonumber
h\, f_{\{m\}}\left(Y_{i}\right)  &=& h\, f_{\{m\}}\left( y(t_{n-1}+c_i h)  \right)  +\mathcal{O}(h^{q+2}) \,, \\
\label{eqn:stage-order-f}
&=& h\, f_{\{m\}}\left( y(t_{n-1}+c_i h)  \right)  +\mathcal{O}(h^{p+1}) \,
\end{eqnarray}
where we have used the fact that $q+2 \ge p+1$. 
Consider the solution step of the individual method $m$ with exact solution arguments
\begin{eqnarray}
\label{eqn:help-4}
\sum_{k=0}^p q_{\{m\}\,i,k} h^k y^{(k)} (t_n) &=& h\sum_{j=1}^s b_{\{m\}i,j} \, f \left(\, y(t_{n-1}+c_j h )\, \right) \\
\nonumber  && + \sum_{j=1}^r v_{\{m\}i,j} \,  \left( \sum_{k=0}^p q_{\{m\}\,i,k} h^k y^{(k)} (t_{n-1})  \right) + \mathcal{O}\left(h^{p+1}\right)
\end{eqnarray}
for $i=1,\dots,r$,
where the size of the error term reflects the fact that each individual method has order $p$.

Use \eqref{eqn:stage-order-f} and the assumption \eqref{eqn:GLM-order_assumption_partitioned} into the $m$-th component of \eqref{eqn:help-4}
to obtain
\begin{eqnarray}
\label{eqn:help-5}
\sum_{k=0}^p q_{\{m\}\,i,k} h^k y^{(k)} (t_n) &=& h\sum_{j=1}^s b_{\{m\}i,j} \, f \left(\, Y_j )\, \right) \\
\nonumber  && + \sum_{j=1}^r v_{\{m\}i,j} \,  y_{\{m\}\,j}^{[n-1]}+ \mathcal{O}\left(h^{p+1}\right)\,, \\
&=&  y_{\{m\}i}^{[n]} + \mathcal{O}\left(h^{p+1}\right)\, \quad i=1,\dots,r,
\end{eqnarray}
The last equality follows from the partitioned method solution equation \eqref{eqn:partitoned_glm_final}.
This  establishes the order $p$ of the partitioned GLM.

\end{proof}

\section{Implicit-explicit general linear methods}\label{sec:IMEX-schemes}


\subsection{Construction procedure}\label{sec:IMEX-construction}

The derivation of IMEX-GLM schemes relies on the partitioned GLM theory 
developed in Section \ref{sec:partitioned-GLM}. 
We transform the additively partitioned system \eqref{eqn:ode_imex} into a component partitioned system \eqref{eqn:partitioned-ode}
via the following transformation \cite{Ascher_1997} 
\begin{subequations}
\begin{eqnarray}
\nonumber
y &=& x+z\,, \\
\label{eqn:nonstiff}
x' &=& \tilde{f}(x,z) = f(x+z) \,,\\
\label{eqn:stiff}
z' &=& \tilde{g}(x,z) = g(x+z)\,.
\end{eqnarray}
\end{subequations}
Equation (\ref{eqn:nonstiff}) is discretized with an explicit (type 1) GLM
\begin{subequations}
\label{eqn:ex_glm}
\begin{eqnarray}
X_i &=& h\sum_{j=1}^{i-1} {a}_{i,j} \, f(X_j+Z_j)+  \sum_{j=1}^r u_{i,j}\, x_j^{[n-1]}, \quad i=1,\dots,s,~\\
x_i^{[n]} &=& h\sum_{j=1}^s {b}_{i,j} \, f(X_j+Z_j) + \sum_{j=1}^r {v}_{i,j} \, x_j^{[n-1]}, \quad i=1,\dots,r\,.
\end{eqnarray}
\end{subequations}
Similarly, equation (\ref{eqn:stiff}) is discretized with an implicit (type 2)  GLM
\begin{subequations}
\label{eqn:im_glm}
\begin{eqnarray}
 & & Z_i = h\sum_{j=1}^{i} \widehat{a}_{i,j} \, g(X_j+Z_j)+   \sum_{j=1}^r \widehat{u}_{i,j}\, z_j^{[n-1]}, \; i=1,\dots,s,~ \\
& & z_i^{[n]}= h\sum_{j=1}^s \widehat{b}_{i,j} \, g(X_j+Z_j) + \sum_{j=1}^r \widehat{v}_{i,j} \, z_j^{[n-1]}, \; i=1,\dots,r.~
\end{eqnarray}
\end{subequations}
Combining (\ref{eqn:ex_glm}) and (\ref{eqn:im_glm}) we obtain
\begin{subequations}
\label{eqn:com_glm}
\begin{eqnarray}
 X_i+Z_i  &=&  h \left( \sum_{j=1}^{i-1} {a}_{i,j} \, f(X_j+Z_j) + \sum_{j=1}^{i} \widehat{a}_{i,j} \, g(X_j+Z_j)  \right)  \\
 \nonumber && +  \sum_{j=1}^r \left( u_{i,j}\,  x_j^{[n-1]} + \widehat{u}_{i,j}\, z_j^{[n-1]} \right) \,, ~~ i=1,\dots,s,~ \\
 x_i^{[n]} + z_i^{[n]}  &=&  h \left( \sum_{j=1}^s {b}_{i,j} \, f(X_j+Z_j)  +  \sum_{j=1}^s \widehat{b}_{i,j} \, g(X_j+Z_j) \right)  \\
 \nonumber && + \sum_{j=1}^r \left( {v}_{i,j} \, x_j^{[n-1]}  + \widehat{v}_{i,j} \, z_j^{[n-1]} \right)\,, ~~ i=1,\dots,r,~
\end{eqnarray}
\end{subequations}
We consider pairs of explicit \eqref{eqn:ex_glm} and implicit \eqref{eqn:im_glm} schemes that 
\begin{itemize}
\item share the same abscissa vector $\mathbf{c}=\widehat{\mathbf{c}}$ such that the partitioned GLM is internally consistent, and
\item share the same coefficient matrices $\mathbf{U}=\widehat{\mathbf{U}}$ and $\mathbf{V}=\widehat{\mathbf{V}}$.
\end{itemize}
For this class of schemes all internal stage vectors can be combined.
Specifically, let $Y_i=X_i+Z_i$ and $y_i = x_i+z_i$. The scheme \eqref{eqn:com_glm} becomes the following method/

\begin{definition}[IMEX-GLM methods]
One step of an implicit-explicit general linear method applied to \eqref{eqn:ode_imex} advances the solution using
\begin{subequations}
\label{eqn:imex_glm}
\begin{eqnarray}
Y_i  &=&  h  \sum_{j=1}^{i-1} {a}_{i,j} \, f(Y_j) + h  \sum_{j=1}^{i} \widehat{a}_{i,j} \, g(Y_j)   + \sum_{j=1}^r u_{i,j} \,   y_j^{[n-1]} \,, ~~i=1,\dots,s,~ \\
y_i^{[n]}  &=&   h\sum_{j=1}^s  \left(   {b}_{i,j} \, f(Y_j) +  \widehat{b}_{i,j} \, g(Y_j)  \right)  +  \sum_{j=1}^r v_{i,j} \, y_j^{[n-1]}  \,, ~~ i=1,\dots,r\,.
\end{eqnarray}
\end{subequations}
\end{definition}
We note that in (\ref{eqn:imex_glm}) $x_i^{[n]}$ and $z_i^{[n]}$ need not to be known individually once they are initialized ine the first step.  The combined solution $y_i^{[n]}=x_i^{[n]}+z_i^{[n]}$ 
is advanced at each step as regular GLMs do. 
The IMEX-GLM \eqref{eqn:imex_glm} is represented compactly by the Butcher tableau
\begin{equation}
\label{eqn:imex_glm_tableau}
  \begin{array}{c|c|c|c}
\mathbf{c} &   \mathbf{A} & \widehat{\mathbf{A}} & \mathbf{U} \\ \hline 
   & \mathbf{B} & \vspace{2pt} \widehat{\mathbf{B}} & \mathbf{V}
\end{array}\,.
\end{equation}
%

\subsection{Starting procedures}\label{sec:IMEX-starting}
An IMEX GLM \eqref{eqn:imex_glm} of order $p$ requires a starting procedure that approximates linear combinations of derivatives as follows
\begin{equation}
\label{eqn:initialization-formula}
x_i^{[0]} = \sum_{k=0}^r {q}_{i,k} h^k x^{(k)}(t_0)  + \mathcal{O}(h^{p}) \quad \textnormal{and} \quad 
z_i^{[0]} = \sum_{k=0}^r \widehat{q}_{i,k} h^k z^{(k)}(t_0) + \mathcal{O}(h^{p})
\end{equation}
respectively, where 
\begin{eqnarray}
\label{eqn:initial-weights}
{q}_0 =  \mathbf{1}_s, \quad {q}_i = \frac{\mathbf{c}^i}{i!} - \frac{{\mathbf{A}}\, \mathbf{c}^{i-1}}{(i-1)!}\, ; \quad 
 \widehat{q}_0 =  \mathbf{1}_s, \quad \widehat{q}_i = \frac{\mathbf{c}^i}{i!} - \frac{\widehat{\mathbf{A}}\, \mathbf{c}^{i-1}}{(i-1)!} .
\end{eqnarray}
Thus
\begin{eqnarray*}
\label{eqn:xpy}
y_i^{[0]} &=& x_i^{[0]}+z_i^{[0]} \\
&=& x(t_0)+z(t_0) + {q}_{i,1} h x'(t_0) + \widehat{q}_{i,1} h z'(t_0) \\
&& + \sum_{k=2}^{r} {q}_{i,k} h^k x^{(k)}(t_0) + \sum_{k=2}^{r} \widehat{q}_{i,k} h^k z^{(k)}(t_0) \\
&=& y_0 + {q}_{i,1} h f(y_0) + \widehat{q}_{i,1} h g(y_0) \\
&& + \sum_{k=2}^{r} {q}_{i,k} h^k x^{(k)}(t_0) + \sum_{k=2}^{r} \widehat{q}_{i,k} h^k z^{(k)}(t_0).
\end{eqnarray*}
Evaluation of the first three terms is straightforward. 
But approximations of the other terms containing derivatives $x^{(k)}(t_0)$ and $y^{(k)}(t_0)$ for $k \ge 2$
requires additional work if their analytical expressions are difficult to obtain.  

To initialize an IMEX GLM we approximate {\it independently} the vectors $h^k x^{(k)}(t_0)$, $h^k z^{(k)}(t_0)$, $k=1,\dots,r$, 
using finite differences and the solution information provided by several steps of an IMEX Runge-Kutta method. 

%
For better accuracy, the IMEX RK method uses a small step size $\tau < h$, and produces the numerical solutions $y^{\rm start}_i \approx y(t_0 + i\tau)$. 
In the following we show how to compute the terms $\tau^k x^{(k)}(t_0)$; each of these terms is then rescaled by $(h/\tau)^k$ to reflect the integration step $h$.  
We have that
\begin{equation}
\vect{\tau x'(t_0) \\ \tau^2 x''(t_0)\\ \vdots \\ \tau^r x^{(r)}(t_0)} =  \tau \mathbf{D} \vect{x'(t_0) \\ x'(t_1)\\ \vdots \\ x'(t_r)} + \mathcal{O}(\tau^{r+1}) = \tau \mathbf{D} \vect{f(y_0) \\ f\left(y^{\rm start}_1\right)\\ \vdots \\ f\left(y^{\rm start}_r\right)} + \mathcal{O}(\tau^{r+1})
\end{equation}
where the coefficient matrix $D \in \mathbb{R}^{r \times r}$ is derived by expanding the right hand side in Taylor series and comparing the coefficients of each term. 
For the cases $r=2$ and $r=3$ the coefficients are
\begin{displaymath}
 \mathbf{D}_{(r=2)}=\left[
\begin{tabular}{c c}
   $1$ & $0$  \\ 
   $-1$ & $1$ 
\end{tabular}
\right]\quad \textnormal{and} \quad
 \mathbf{D}_{(r=3)}=\left[
\begin{tabular}{c c c}
   $1$ & $0$ & $0$  \\ 
   $-3/2$ & $2$ & $-1/2$ \\
   $1$ & $-2$ & $1$
\end{tabular}
\right]\,,
\end{displaymath}
respectively.
The same procedure is applied to obtain $\tau^k z^{(k)}(t_0)$. We note that the initialization procedure requires the function values $f(y)$ and $g(y)$ 
evaluated at the starting solution steps $y^{\rm start}_i$, and that there is no need to compute $x_i$ or $z_i$ separately.

\subsection{Termination procedures}\label{sec:IMEX-termination}

To generate the solution at the last time step $y(t_F)$ using (\ref{eqn:solution}) a general termination procedure reads 
\begin{subequations}
\label{eqn:imex_glm_termination}
\begin{eqnarray}
\label{eqn:imex_glm_termination_full}
y(t_n) &\approx& \sum_{i=0}^s \beta_{0,i} f(Y_i) + \sum_{j=0}^r \gamma_{0,j} x_j^{[n-1]} \\
\nonumber
 && + \sum_{i=0}^s \widehat{\beta}_{0,i} g(Y_i) + \sum_{j=0}^r \widehat{\gamma}_{0,j} z_j^{[n-1]} \,.
\end{eqnarray}
In order to avoid separate evaluations of $x_j^{[n-1]}$ and $z_j^{[n-1]}$ we require that $\gamma_{0,j}=\widehat{\gamma}_{0,j}$ for all $j$.
In this case the termination procedure reads
\begin{eqnarray}
\label{eqn:imex_glm_termination_reduced}
y(t_n) &\approx& \sum_{i=0}^s \beta_{0,i} f(Y_i)  + \sum_{i=0}^s \widehat{\beta}_{0,i} g(Y_i) + \sum_{j=0}^r \gamma_{0,j} y_j^{[n-1]}  \,.
\end{eqnarray}
\end{subequations}

For explicit (type 1) GLMs, choosing the first abscissa coordinate $c_1 = 0$ implies that  $q_{1,0} = 1$ and $q_{1,i} =0 $ for $ i \ge 1$ due to order conditions.
The first element of the output vector is exactly the solution at the current step, $y^{[n]}_1 \approx y(t_n)$. 
In this case, $\beta_0$ is equal to the first row of the coefficient matrix $\mathbf{B}$, and $\gamma_0$ is the first row of $\mathbf{V}$. 

For implicit (type 2) GLMs, there are usually sufficiently many free parameters in $\widetilde{\mathbf{B}}$ and $\widetilde{\mathbf{V}}$ that remain after satisfying (\ref{eqn:coeB}). 
These free parameters could be chosen in such a way that the implicit GLM shares the same coefficients $\gamma_0$ with the explicit GLM.
The difficulty of computing terms $x_j^{[n-1]}$ and $z_j^{[n-1]}$ individually can therefore be avoided. 

\subsection{Linear stability analysis}\label{sec:IMEX-stability}

For convenience, we write the IMEX-GLM (\ref{eqn:imex_glm}) in the vector form
\begin{subequations}
\label{eqn:IMEX_DIMSIM_VECTOR}
\begin{eqnarray}
Y &=& h \mathbf{A} F(Y) + h \widehat{\mathbf{A}} G(Y) + \mathbf{U}\, y^{[n-1]} \\ 
y^{[n]} &=& h \mathbf{B} F(Y) + h \widehat{\mathbf{B}} G(Y) + \mathbf{V}\, y^{[n-1]}\,.
\end{eqnarray}
\end{subequations}
We consider the generalized linear test equation
\begin{equation}
\label{eqn:test-equation}
 y'= \xi y + \widehat{\xi} y\,, \quad t \geq 0,
\end{equation}
where $\xi$ and $\widehat{\xi}$ are complex numbers. We consider $\xi y$ to be the nonstiff term and $\widehat{\xi} y$ the stiff term, and denote
 $w=h \xi$ and $\widehat{w}=h \widehat{\xi}$.

Applying (\ref{eqn:IMEX_DIMSIM_VECTOR}) to the test equation \eqref{eqn:test-equation} leads to
\begin{subequations}
 \begin{eqnarray}
 Y &=& h \left( \xi \mathbf{A} + \widehat{\xi} \widehat{\mathbf{A}} \right) Y +  \mathbf{U} \, y^{[n-1]}, \\
   y^{[n]} &=& h \left( \xi \mathbf{B} + \widehat{\xi} \widehat{\mathbf{B}} \right)  Y + \mathbf{V} \, y^{[n-1]} \,.
 \end{eqnarray}
\end{subequations}
Assuming $\mathbf{I}_{s \times s}- w \mathbf{A} - \widehat{w} \widehat{\mathbf{A}}$ is nonsingular we obtain 
\begin{displaymath}
 y^{[n]} = \mathbf{M}(w,\widehat{w})\, y^{[n-1]},
\end{displaymath}
where  the stability matrix is defined by
\begin{equation}
 \mathbf{M}(w,\widehat{w}) = \mathbf{V} + \left( w\, \mathbf{B}+ \widehat{w}\, \widehat{\mathbf{B}} \right) \left(\mathbf{I}_{s \times s}-w\, \mathbf{A} - \widehat{w}\,\widehat{\mathbf{A}}\right)^{-1}
 \, \mathbf{U}\,.
\end{equation}
Let $S\subset \mathbb{C}$ and $\widehat{S}\subset \mathbb{C}$ be the stability regions of the explicit GLM and of the implicit GLM, respectively.
The {\it combined stability region} is defined by
\begin{equation}
 \label{eqn:stab_region}
\left\{\, w \in S, \, \widehat{w} \in \widehat{S}~:~\rho\bigl( \mathbf{M}(w,\widehat{w}) \bigr) \le 1 \, \right\} \subset S \times \widehat{S} \subset \mathbb{C} \times \mathbb{C}\,.
\end{equation}
For a practical analysis of stability we define a {\it desired stiff stability region}, e.g.,
\[
\widehat{\mathcal{S}}_\alpha  = \{ \widehat{w} \in \widehat{S} \cap \mathbbm{C}^-~:~ |\mbox{Im}(\widehat{w})| \le \tan(\alpha)\, |\mbox{Re}(\widehat{w})| \}\,,
\]
and compute numerically the corresponding non-stiff stability region:
\begin{equation}
\label{eqn:stability-constrained}
\mathcal{S}_\alpha = \left\{ w \in S ~:~ \rho\bigl( \mathbf{M}(w,\widehat{w}) \bigr) \le 1\,,~~\forall\, \widehat{w} \in \widehat{\mathcal{S}}_\alpha \right\}\,.
\end{equation}
The IMEX-GLM method is stable if the constrained non-stiff stability region $\mathcal{S}_\alpha$ is non-trivial (has a non-empty interior) and is sufficiently
large for a prescribed (problem-dependent) value of $\alpha$, e.g., $\alpha = 90^\circ$.

\subsection{Prothero-Robinson convergence}\label{sec:PR-convergence}

We now study the possible order reduction for very stiff systems. 
We consider the Prothero-Robinson (PR) \cite{Prothero_1974} test problem written as a split system \eqref{eqn:ode_imex} 
\begin{equation}
y' = \underbrace{ \mu\, (y - \phi(t)) }_{g(y)} + \underbrace{\phi'(t) }_{f(y)} ~, \quad \mu < 0~, \quad y(0)=\phi(0)~,
\label{Prothero-Robinson}
\end{equation}
where the exact solution is $y(t)=\phi(t)$. A numerical method is said to be PR-convergent with order $p$
if its application to (\ref{Prothero-Robinson}) gives a solution whose the global error decreases as $\mathcal{O}(h^p)$
for $h \rightarrow 0$ and $h\mu \rightarrow -\infty$.

\begin{theorem}[Prothero-Robinson convergence of IMEX-GLM]
\label{thm:PR-convergence}
Consider the IMEX GLM method (\ref{eqn:imex_glm}). Without loss of generality we consider that $\mathbf{U}=\mathbf{I}$.
The explicit part is of order $p$ and stage order $q \in \{ p-1,p\}$, 
and the implicit part has order $\widehat{p}=p$ and stage order $\widehat{q} \in \{ p-1,p\}$. 
Assume that 
$h\mu \in \widehat{S}$ for all $h > 0$.
Then the IMEX GLM method (\ref{eqn:imex_glm}) is PR-convergent with order $\min(p,q)$.
\end{theorem}

\begin{remark}
If the explicit stage order is $q = p$, then the PR order of convergence is $p$. It is convenient
to construct IMEX GLM methods (\ref{eqn:imex_glm}) with explicit stage order $q=p$, even if $\widehat{q}=p-1$,
as such methods do not suffer from stiff order reduction on the PR problem.
\end{remark}

\begin{proof}
Let
\[
\phi^{[n]} =  \phi\left(t_{n-1} + \mathbf{c}\, h \right) = \left[  \phi(t_{n-1} +c_1\, h ),\ldots,\phi(t_{n-1} + c_s\, h ) \right]^T\,.
\]
and
\[
\psi^{[n]} =   \left[  \phi(t_{n-1}),h\, \phi'(t_{n-1}),\ldots,h^p\, \phi^{(p)}(t_{n-1}) \right]^T\,.
\]
The method \eqref{eqn:imex_glm} applied to (\ref{Prothero-Robinson}) reads:
\begin{subequations}
\label{eqn:imex_glm_PR}
\begin{eqnarray}
\label{eqn:imex_glm_PR_stage}
 Y^{[n]} &=&  h \,{\mathbf{A}} \, \phi'^{[n]}  + h\, \mu\, \widehat{\mathbf{A}} \,  \left(Y^{[n]} - \phi^{[n]} \right)  + \mathbf{U}\,  y^{[n-1]} \,,  \\
\label{eqn:imex_glm_PR_final}
y^{[n]}  &=&   h \, {\mathbf{B}}\,  \phi'^{[n]}  + h\, \mu\, \widehat{\mathbf{B}} \,  \left(Y^{[n]} - \phi^{[n]} \right)  +  \mathbf{V} \, y^{[n-1]}\,.
\end{eqnarray}
\end{subequations}
Consider the global stage errors
\begin{eqnarray*}
E^{[n]} &=& Y^{[n]} - \phi^{[n]}\,.
\end{eqnarray*}
To obtain the global error in $y^{[n]}$ we consider separately the global errors in the nonstiff and stiff components:
\begin{eqnarray*}
e^{\rm nonstiff}_n &=& x^{[n]} - \sum \mathbf{q}_k \, h^k\, x^{(k)}(t_n)\,, \\
e^{\rm stiff}_n &=& z^{[n]} - \sum \widehat{\mathbf{q}}_k \, h^k\, z^{(k)}(t_n)~, \\
 &=& \phi^{[n]}-x^{[n]} - \sum \widehat{\mathbf{q}}_k \, h^k\, \left( \phi^{(k)} - x^{(k)} \right) (t_n) \\
 &=& \phi^{[n]}-x^{[n]} 
\end{eqnarray*}
since the exact solution of the nonstiff system is $x(t) = \phi(t)$.
Consequently, the total error is
\begin{eqnarray*}
e_n &=& e^{\rm nonstiff}_n + e^{\rm stiff}_n \\
&=& \phi^{[n]}  - \sum \mathbf{q}_k \, h^k\, \phi^{(k)}(t_n) \\
&=& \phi^{[n]} - \mathbf{W}\, \psi^{[n]}\,.
\end{eqnarray*}
Write the stage equation \eqref{eqn:imex_glm_PR_stage} in terms of the exact solution and global errors
\begin{eqnarray*}
 E^{[n]} + \phi^{[n]} &=&  h \,\mathbf{A} \, \phi'^{[n]}  + h\, \mu\, \widehat{\mathbf{A}} \, E^{[n]} + e_{n-1} + \mathbf{U}\,\sum_{k=0}^p \mathbf{q}_k \, h^k\, \phi^{(k)}(t_{n-1}) \,,  
\end{eqnarray*}
to obtain
\begin{eqnarray}
\label{eq:imex_glm-stage-error}
\left( \mathbf{I}_{s \times s} - h\,\mu\, \widehat{\mathbf{A}} \right)\, E^{[n]}&=& e_{n-1} 
+h\, \mathbf{A}\, \phi'\left(t_{n-1} + \mathbf{c}\, h \right) \\
\nonumber
&& \displaystyle + \mathbf{U}\, \sum_{k=0}^p \mathbf{q}_k \, h^k\, \phi^{(k)}(t_{n-1}) - \phi(t_{n-1}+\mathbf{c} h)\,.
 \end{eqnarray}
The exact solution is expanded in Taylor series about $t_{n-1}$:
\begin{eqnarray*}
\phi\left(t_{n-1} + \mathbf{c}\, h \right)-\mathbf{1}_s\,\phi(t_{n-1}) &=& \sum_{k=1}^\infty \frac{h^k \mathbf{c}^k }{k!}\phi^{(k)}(t_{n-1}) \,,\\
h\, \phi'\left(t_{n-1} + \mathbf{c}\, h \right) &=& \sum_{k=1}^\infty \frac{k h^k \mathbf{c}^{k-1} }{k!}\phi^{(k)}(t_{n-1}) \,.
\end{eqnarray*}
Inserting the above Taylor expansions in \eqref{eq:imex_glm-stage-error} leads to
\begin{eqnarray*}
\left( \mathbf{I}_{s \times s} - h\,\mu\, \widehat{\mathbf{A}} \right)\, E^{[n]}&=& e_{n-1} -\mathbf{1}_s\,\phi(t_{n-1}) + \mathbf{U}\mathbf{q}_0\,\phi(t_{n-1}) \notag\\
&&+\sum_{k=1}^\infty \left( k\, \mathbf{A}\,\mathbf{c}^{k-1} + k!\, \mathbf{U}\,\mathbf{q}_k - \mathbf{c}^k \right)\, \frac{h^k }{k!}\phi^{(k)}(t_{n-1})\notag\\
&=& e_{n-1} + {\cal O}\left(h^{q+1}\right)
\end{eqnarray*}
where $q$ is the stage order of the explicit method. 
We have used the facts that $\mathbf{q}_0=\mathbf{1}_s$, $ \mathbf{U}\,\mathbf{1}_s=\mathbf{1}_s$, and the order conditions \eqref{eqn:GLM-order_condition-1}
and \eqref{eqn:GLM-order_condition-1-b} for the cases where $q=p$ and $q=p-1$, respectively.

Similarly, we write the solution equation \eqref{eqn:imex_glm_PR_final} in terms of the exact solution and global errors:
\begin{eqnarray*}
e_n + \sum_{k=0}^p \mathbf{q}_k \, h^k\, \phi^{(k)}(t_n)   &=&   h \, \mathbf{B}\,  \phi'(t_{n-1} + \mathbf{c}h)  + h\, \mu\, \widehat{\mathbf{B}} \,  E^{[n]}  +  \mathbf{V} \, e_{[n-1]} \\
&& + \mathbf{V} \sum_{k=0}^p \mathbf{q}_k \, h^k\, \phi^{(k)}(t_{n-1})  \,.
\end{eqnarray*}
After rearranging the expression we obtain
\begin{eqnarray*}
e_n  &=&  \left( h\, \mu\, \widehat{\mathbf{B}} \, \left( \mathbf{I}_{s \times s} - h\,\mu\, \widehat{\mathbf{A}} \right)^{-1} + \mathbf{V} \right)\, e_{n-1}  \\
&& + h \, \mathbf{B}\,  \phi'(t_{n-1} + \mathbf{c}h)  + \mathbf{V} \sum_{k=0}^p \mathbf{q}_k \, h^k\, \phi^{(k)}(t_{n-1}) \\
&& - \sum_{k=0}^p \mathbf{q}_k \, h^k\, \phi^{(k)}(t_n)  + {\cal O}\left(h^{q+1}\right) \,.
\end{eqnarray*}
By Taylor series expansion we have
\[
\sum_{k=0}^p \mathbf{q}_k \, h^k\, \phi^{(k)}(t_n) = \sum_{k=0}^p \left(  \sum_{\ell=0}^k \frac{\mathbf{q}_{k-\ell}}{\ell!} \right) h^k \, \phi^{(k)}(t_{n-1})
\]
and therefore
\begin{eqnarray}
e_n 
&=&\widehat{\mathbf{M}}(h\mu) \, e_{n-1}  \notag\\
&& +  \sum_{k=1}^\infty \left( k\,\mathbf{B} \,\mathbf{c}^{k-1}  + k! \, \mathbf{V}\, \mathbf{q}_k - k!   \sum_{\ell=0}^k \frac{\widehat{q}_{k-\ell}}{\ell!}  \right) \, \frac{h^k}{k!}\phi^{(k)}(t_{n-1}) + {\cal O}\left(h^{\widehat{q}+1}\right) \notag\\
\label{eqn:PR-solved-solution-err}
\end{eqnarray}
The order condition \eqref{eqn:GLM-order_condition-2} of the nonstiff scheme reads
\begin{eqnarray*}
&& e^z w(z) = z \mathbf{B}\, e^{cz} + \mathbf{V} w(z) + \mathcal{O}\left(z^{p+1}\right) \\
&& \sum_{\ell \ge 0} \sum_{k=0}^p  \frac{\mathbf{q}_k z^{k+\ell}}{\ell!} =  \sum_{k=0}^\infty  \mathbf{B} \frac{\mathbf{c}^k z^{k+1}}{k!} 
+  \sum_{k=0}^p \mathbf{V} \mathbf{q}_k z^k + \mathcal{O}\left(z^{p+1}\right)\,.
\end{eqnarray*}
Identification of powers of $z^k$ leads to
\begin{eqnarray*}
&& \sum_{\ell = 0}^k p  \frac{\widehat{q}_{k-\ell} z^{k}}{\ell!} =  \mathbf{B} \frac{\mathbf{c}^{k-1} z^{k}}{(k-1)!} 
+  \mathbf{V} \mathbf{q}_k z^k\,, \quad k=1,\dots,p\,.
\end{eqnarray*}
%
%
The error recurrence \eqref{eqn:PR-solved-solution-err} becomes
\begin{eqnarray}
\label{eqn:PR-error-reccurence-mod}
e_n 
&=&\widehat{\mathbf{M}}( h\mu) \, e_{n-1}  + {\cal O}\left(h^{\min(q+1,p+1)}\right) \,.
\end{eqnarray}
Assume that the initial error is $e_0 = \mathcal{O}(h^p)$.
The error amplification matrix $\widehat{\mathbf{M}}(h\mu)$ is the 
stability matrix of the implicit method. 
Therefore its spectral radius is uniformly bounded below one for all argument values $h\mu$ of interest.
By standard numerical ODE arguments \cite{Hairer_1993} the
equation (\ref{eqn:PR-error-reccurence-mod}) implies convergence of global errors to zero at a rate 
$\| e_n \| = {\cal O}\left(h^{\min(p,q)}\right)$.
\end{proof}

\section{Construction of implicit-explicit methods of orders two and three}

We now construct IMEX-DIMSIM methods as summarized in Section \ref{sec:DIMSIM}.
Specifically,  we focus on DIMSIMs with $p=q=r=s$, $\mathbf{U}= \mathbf{I}_{s \times s}$, and $\mathbf{V}=\mathbf{1}_s\, v^T$, where $v^T\, \mathbf{1}_s=1$ \cite{Jackiewicz_2009_book}.

\subsection{Two-stage, second-order pairs with $p=q=r=s=2$}\label{sec:IMEX-order2}

The pair of explicit and implicit schemes developed in \cite{Zhang_2012} is named IMEX-DIMSIM-2A and consists of a type 2 DIMSIM from \cite{Butcher_1993a} with the same stability of SDIRK method of order 2, and a type 1 derived DIMSIM. Both of them share the same abscissa vector $\mathbf{c}=[0,1]^T$ and the same coefficient matrix $\mathbf{V}$. The IMEX-DIMSIM-2A coefficients in the tableau \eqref{eqn:imex_glm_tableau} representation are 
\begin{equation*}
\label{eqn:imex_glm_ord2}
  \begin{array}{c|c c |c c |c c}
    0 &    0 &      0 & \frac{2-\sqrt{2}}{2} &    0     &     1  & 0  \\  
    1 &    2 &      0 &   \frac{2\sqrt{2}+6}{7} & \frac{2-\sqrt{2}}{2} & 0 & 1\\ \hline
      &   \frac{3\sqrt{2}-1}{4} & \frac{3-\sqrt{2}}{4} & \frac{73-34\sqrt{2}}{28} & \frac{4\sqrt{2}-5}{4} & \frac{3\sqrt{2}-3}{4} & \frac{1-\sqrt{2}}{4} \\
      &   \frac{3\sqrt{2}-3}{4} & \frac{1-\sqrt{2}}{4} & \frac{87-48\sqrt{2}}{28} & \frac{-45+34\sqrt{2}}{28} &
\frac{3-\sqrt{2}}{2} & \frac{\sqrt{2}-1}{2}  
\end{array}\,.
\end{equation*}

The choice of $\lambda = (2-\sqrt{2})/ 2$ ensures the type implicit part of IMEX-DIMSIM-2A is L-stable. 
Inherited Runge-Kutta stability is a desirable property, 
but there are not enough free parameters to enforce this property on both methods of the IMEX pair at the same time. 

For a given implicit scheme we construct the explicit method by maximizing the constrained stability region 
\eqref{eqn:stability-constrained}. We have observed that simply maximizing the
explicit stability region $S$ is insufficient and can lead to a very poor constrained stability region for the IMEX method. 
The matrix $\mathbf{B}$ can be determined by $\mathbf{A}$, $\mathbf{c}$ and $\mathbf{V}$ according to the order condition (\ref{eqn:coeB}). 
The only free parameter is $a_{2,1}$ in matrix $\mathbf{A}$, and it is chosen such as to maximize IMEX stability.
First, we use a Matlab Differential Evolution package \footnote{http://www.mathworks.com/matlabcentral/fileexchange/18593-differential-evolution} as a heuristic for global optimization to generate a starting point. 
Then we run the Matlab routine \texttt{fminsearch} multiple times until the result converges; 
each run is initialized with the previous result. The resulting stability regions are reported in Figure \ref{fig:sta_reg_ord2}.

This procedure led to another explicit scheme that maximizes the IMEX stability
\[
\mathbf{A} = \begin{bmatrix} 0 \quad 0 \\ 1.5 \quad 0\end{bmatrix}, \quad 
\mathbf{B} = \begin{bmatrix} \frac{\sqrt{2}}{2} \quad \frac{3-\sqrt{2}}{4} \\ \frac{\sqrt{2}-1}{2} \quad \frac{3-\sqrt{2}}{4}\end{bmatrix};
\]
$\mathbf{U}$ and $\mathbf{V}$ are the same.
We call the new pair IMEX-DIMSIM-2B.
The termination procedure \eqref{eqn:imex_glm_termination}
has the following parameters
\begin{eqnarray*}
\widehat{\beta}_{0,1} = \widehat{b}_{1,1}, \quad \widehat{\beta}_{0,2} = \widehat{b}_{1,2}, \quad \widehat{\gamma}_{0,1} = v_{1,1}, \quad 
\widehat{\gamma}_{0,2} = v_{1,2}. 
\end{eqnarray*}
Solving the condition (\ref{eqn:BV}) gives
\begin{eqnarray*}
& &\beta_{0,1} = \frac{73-34 \sqrt{2}}{28} + \frac{43-31\sqrt{2}}{28}g , \quad \beta_{0,2} = \frac{-1+2 \sqrt{2}}{4} +\frac{-4+3\sqrt{2}}{4}g , \\
& &\gamma_{0,1} = \frac{3-\sqrt{2}}{2} + \frac{2-\sqrt{2}}{2}g ,\quad \gamma_{0,2} = \frac{\sqrt{2}-1}{2} + \frac{\sqrt{2}-2}{2}g .\\
\end{eqnarray*}
The choice of the free parameter $g = 0$ leads to $\gamma_{0,1} = \widehat{\gamma}_{0,1}$, $\gamma_{0,2} = \widehat{\gamma}_{0,2}$, and \eqref{eqn:imex_glm_termination_reduced}.

\begin{figure} 
\centering{
\subfigure[Stability region $\widehat{S}$ of the implicit method]{
\includegraphics[width=0.3\textwidth]{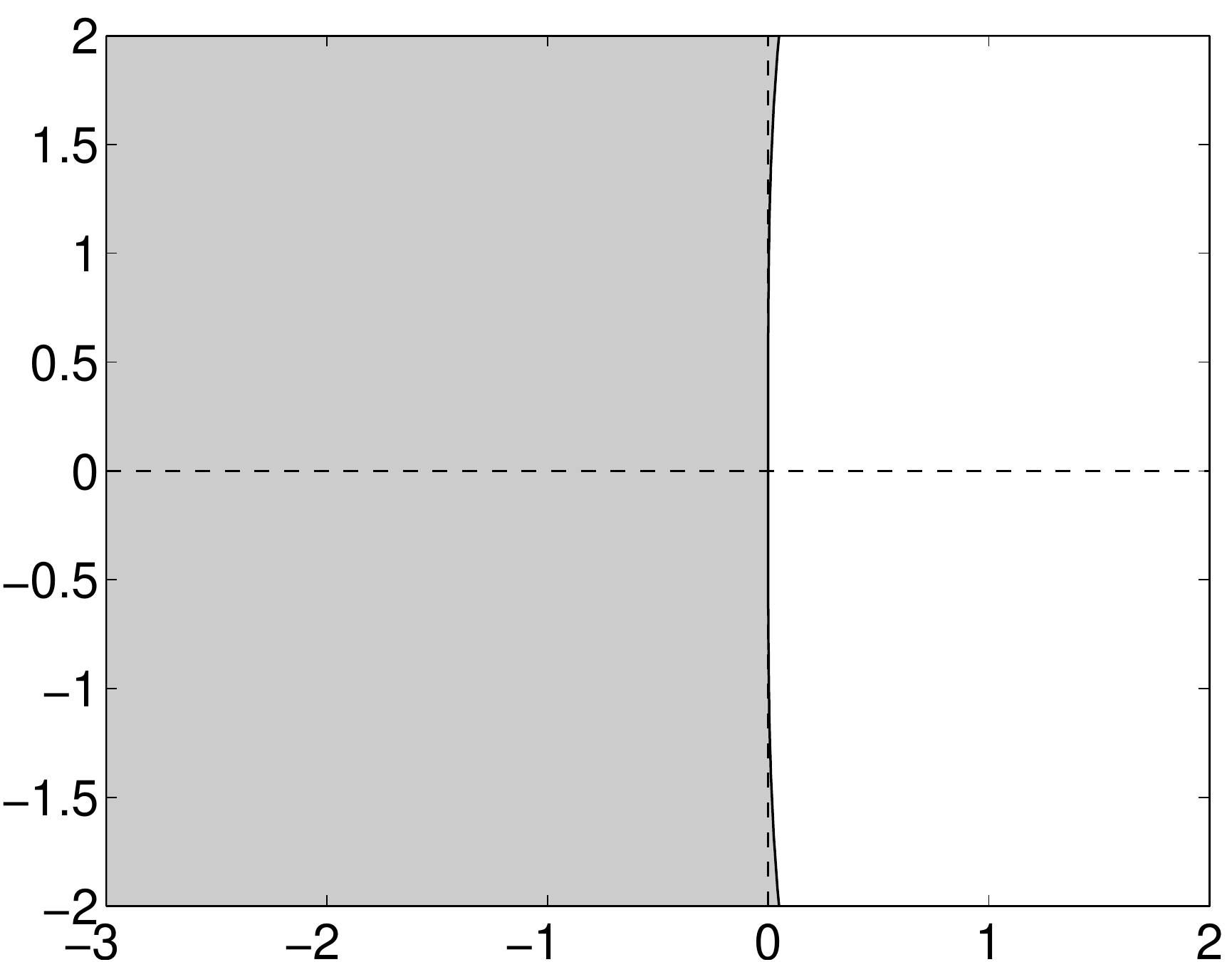}
}
  \subfigure[Stability region $S$ of the explicit method]{
 \includegraphics[width=0.3\textwidth]{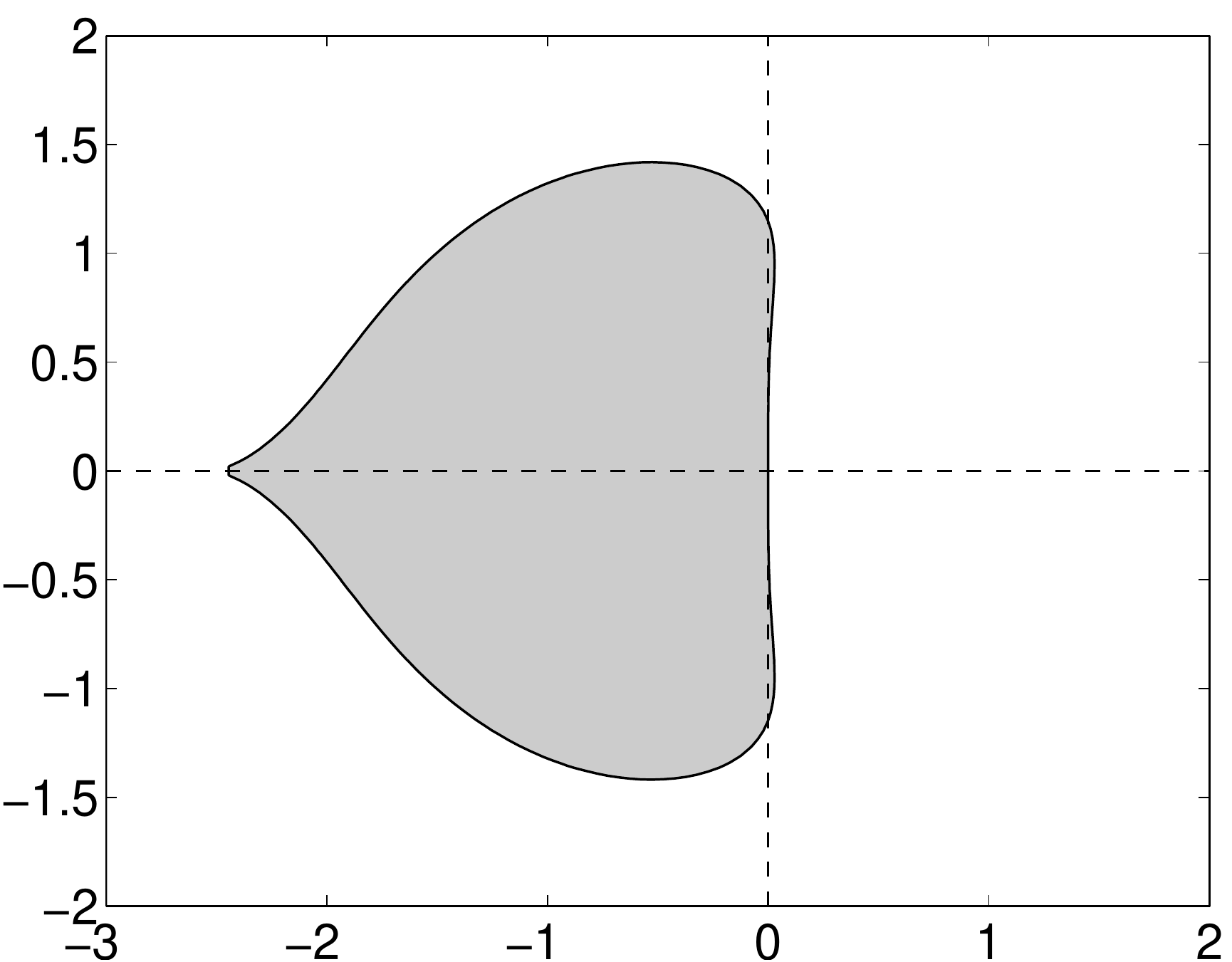}
 }
\subfigure[Constrained stability region $\mathcal{S}_{\alpha}$ \eqref{eqn:stability-constrained} for
 $\alpha=90^\circ$]{
 \includegraphics[width=0.3\textwidth]{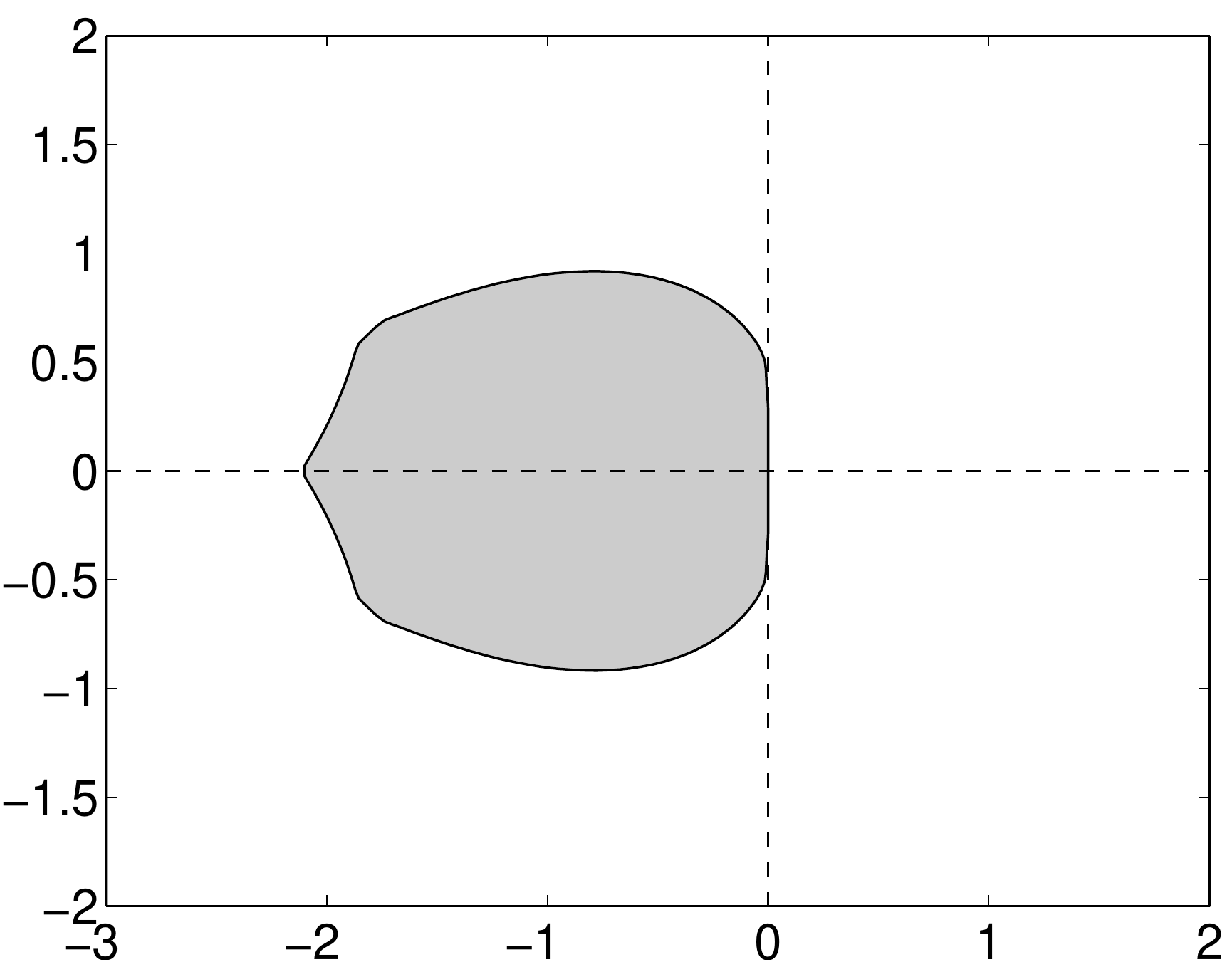}
 }
}
\caption{Stability regions for the IMEX-DIMSIM-2B pair}
\label{fig:sta_reg_ord2}
\end{figure}
%

\subsection{Three-stage, third-order pairs with $p=q=r=s=3$}\label{sec:IMEX-order3a}

We construct two implicit-explicit pairs named IMEX-DIMSIM-3A and IMEX-DIMSIM-3B starting from two existing implicit methods. All coefficients are obtained from the numerical solution of order conditions using Mathematica. The calculations are performed with 24 digits of accuracy such as to reduce the impact of roundoff errors on the 
resulting coefficient values. 
 
\paragraph{IMEX-DIMSIM-3A.} According to \cite{Butcher_1996} there are five A-stable type 2 DIMSIMs with the choice $\lambda=1/2$ and $\mathbf{c}=[0,1/2,1]^T$. 
We select the implicit component in Table \ref{tab:order-3a-implicit-coefficients} which has a balanced set of coefficients.

The explicit component is obtained by a numerical maximization of the constrained stability region,
as discussed in the previous section. The resulting coefficients are shown in Table \ref{tab:order-3a-explicit-coefficients}. 
The IMEX stability regions are drawn in Figure \ref{fig:sta_reg_ord3a}.

The termination procedure  \eqref{eqn:imex_glm_termination} is given by 
\begin{eqnarray*}
& &\beta_{0,1} = 1.01640094894605, \quad \beta_{0,2} = 0.632229903531054, \quad \beta_{03} = 0.0919425241172364, \\
& &\widehat{\beta}_{0,1} = b_{1,1}, \quad \widehat{\beta}_{0,2} = b_{1,2}, \quad \widehat{\beta}_{03} =b_{13}, \\
& &\gamma_{0,1} = \widehat{\gamma}_{0,1} =  v_{1,1}, \quad \gamma_{0,2} = \widehat{\gamma}_{0,2} = v_{1,2}, \quad \gamma_{0,3} = \widehat{\gamma}_{03} = v_{1,3} .
\end{eqnarray*}

\paragraph{IMEX-DIMSIM-3B.} The choice of $\lambda = 0.435866521508459$ and $\mathbf{c}=[0,1/2,1]^T$ leads to the L-stable type 2 DIMSIM reported in \cite{Butcher_1996}. The coefficients of the implicit component are presented in Table \ref{tab:order-3b-implicit-coefficients}.

The type 1 component  is shown in Table \ref{tab:order-3b-explicit-coefficients}. The IMEX stability regions are drawn in Figure \ref{fig:sta_reg_ord3b}.

The coefficients $\widehat{\beta}$ and $\gamma$ of the termination procedure  \eqref{eqn:imex_glm_termination} are equal to the first rows of matrices $\mathbf{B}$ and $\mathbf{V}$, respectively. In addition
\begin{equation*}
\beta_{0,1} = 0.833790728250125, \quad \beta_{0,2} = 0.645998912146314, \quad \beta_{0,3} = 0.120039435995489. 
\end{equation*}
\begin{sidewaystable}[H]
\resizebox{\textheight}{!}{
\tabcolsep=2pt
%
\begin{tabular}{c c c | c c c}
  $0.5$ & $0$ & $0$ & $1$ & $0$ & $0$ \\ 
  $0.200835027145109$ & $0.5$ & $0$ & $0$ & $1$ & $0$ \\ 
  $-1.30998408899641$ & $1.01685248853025$ & $0.5$ & $0$ & $0$ & $1$ \\ \hline
   $1.01640094894605$ & $0.632229903531054$ & $-0.408057475882764$ & $0.910428360600012$ & $0.358564648055175$ & $-0.268993008655188$ \\
  $0.724734282279383$ & $1.46556323686439$ & $-0.6505591694540$ &  $0.910428360600012$ & $0.358564648055175$ & $-0.268993008655188$  \\
  $-0.333784872917534$ & $4.34945403578847$ & $-1.481964185810437$ &  $0.910428360600012$ & $0.358564648055175$ & $-0.268993008655188$  \\
\end{tabular}}
%
\caption{\label{tab:order-3a-implicit-coefficients}Coefficients of the implicit method of the IMEX-DIMSIM-3A pair.}
\end{sidewaystable}
\begin{sidewaystable}[H]
\resizebox{\textheight}{!}{
\tabcolsep=2pt
%
\begin{tabular}{c c c | c c c}
  $0.435866521508459$ & $0$ & $0$ & $1$ & $0$ & $0$ \\ 
  $0.250514880897719$ & $0.435866521508459$ & $0$ & $0$ & $1$ & $0$ \\ 
  $-1.211594287777006$ & $1.00127459988119$ & $0.435866521508459$ & $0$ & $0$ & $1$ \\ \hline
   $0.833790728250125$ & $0.645998912146314$ & $-0.315827085512970$ & $0.552090962040363$ & $0.734856659871292$ & $-0.286947621911655$ \\
  $0.606257540075000$ & $1.28693181000502$ & $-0.479741676094274$ &  $0.552090962040363$ & $0.734856659871292$ & $-0.286947621911655$  \\
  $-0.308416769489771$ & $3.80342155052421$ & $-1.12072253825515$ &  $0.552090962040363$ & $0.734856659871292$ & $-0.286947621911655$  \\
\end{tabular}}
%
\caption{\label{tab:order-3b-implicit-coefficients}Coefficients of the implicit method of the IMEX-DIMSIM-3B pair.}
\end{sidewaystable}
\begin{sidewaystable}
\resizebox{\textheight}{!}{
\tabcolsep=2pt
\begin{tabular}{c c c | c c c}
  $0$ & $0$ & $0$ & $1$ & $0$ & $0$ \\ 
  $0.773142038041842$ & $0$ & $0$ & $0$ & $1$ & $0$ \\ 
  $-0.574721803854933$ & $1.40234019763932$  & $0$ & $0$ & $0$ & $1$ \\ \hline
  $0.568615416356845$ & $0.349254080830621$ & $0.226439028444830$ &  $0.910428360600012$ & $0.358564648055175$ & $-0.268993008655188$    \\
  $0.776948749690179$ & $-0.317412585836046$ & $0.411630323736322$  & $0.910428360600012$ & $0.358564648055175$ & $-0.268993008655188$    \\
  $0.332941885384188$ & $1.22294134041526$ & $-0.239193093951542$& $0.910428360600012$ & $0.358564648055175$ & $-0.268993008655188$ 
\end{tabular}}
\caption{Coefficients of the explicit method of the IMEX DIMSIM-3A pair.\label{tab:order-3a-explicit-coefficients}}
\end{sidewaystable}

\begin{sidewaystable}
\resizebox{\textheight}{!}{
\tabcolsep=2pt
\begin{tabular}{c c c | c c c}
  $0$ & $0$ & $0$ & $1$ & $0$ & $0$ \\ 
  $0.753076872681821$ & $0$ & $0$ & $0$ & $1$ & $0$ \\ 
  $-0.4897243738259477$ & $1.28728279647947$  & $0$ & $0$ & $0$ & $1$ \\ \hline
  $0.755324932592235$ & $0.24363012413977$ & $0.245110297813246$ &  $0.552090962040363$ & $0.734856659871292$ & $-0.286947621911655$    \\
  $0.963658265925568$ & $-0.423036542526896$ & $0.450366758464759$  & $0.552090962040363$ & $0.734856659871292$ & $-0.286947621911655$    \\
  $0.634708802779431$ & $0.772145180244847$ & $0.0396529488674508$& $0.552090962040363$ & $0.734856659871292$ & $-0.286947621911655$
\end{tabular}}
\caption{Coefficients of the explicit method of the IMEX DIMSIM-3B pair.\label{tab:order-3b-explicit-coefficients}}
\end{sidewaystable}

\begin{figure} 
\centering{
\subfigure[Stability region $\widehat{S}$ of the implicit method]{
\includegraphics[width=0.3\textwidth]{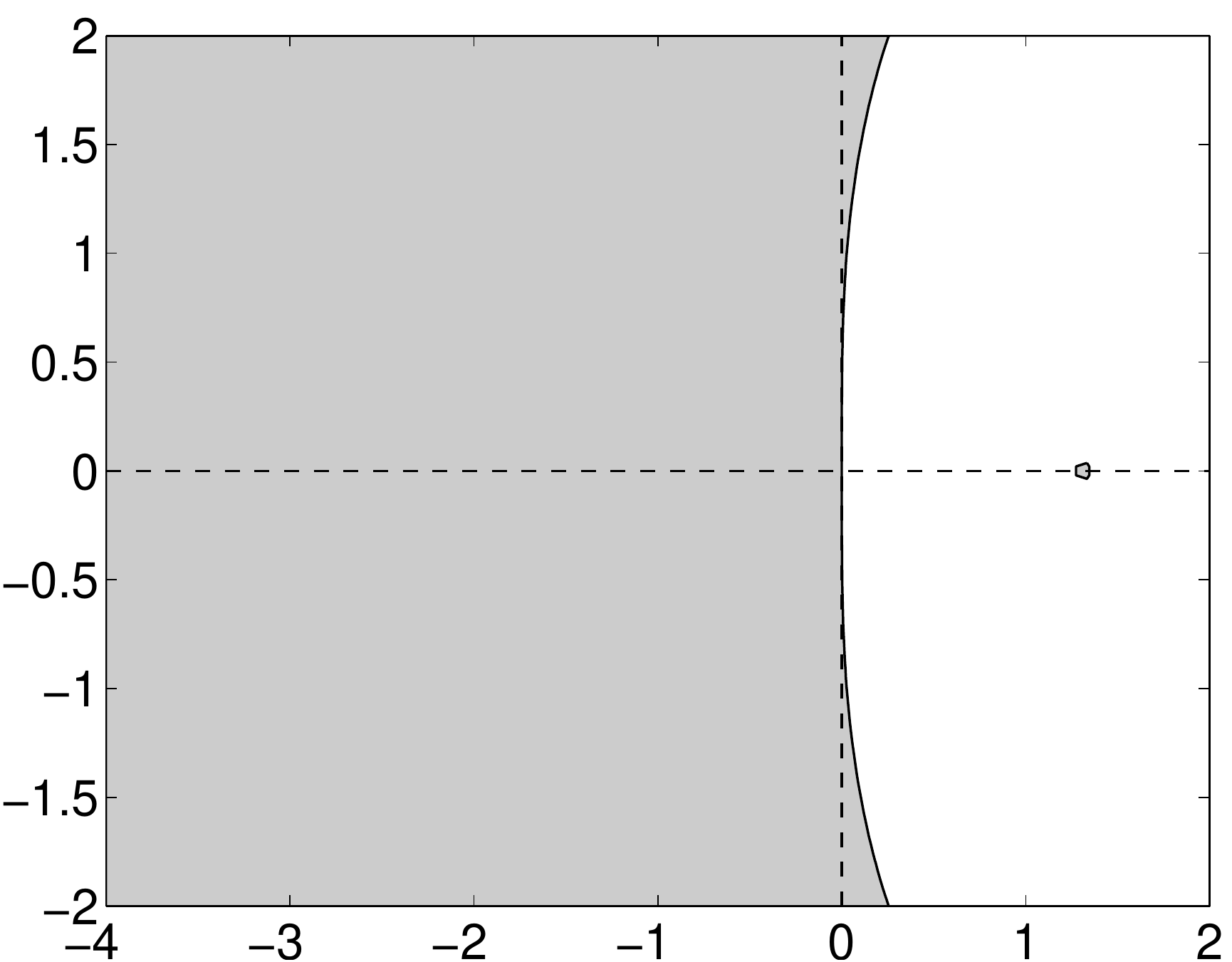}
}
 \subfigure[Stability region $S$ of the explicit method]{
 \includegraphics[width=0.3\textwidth]{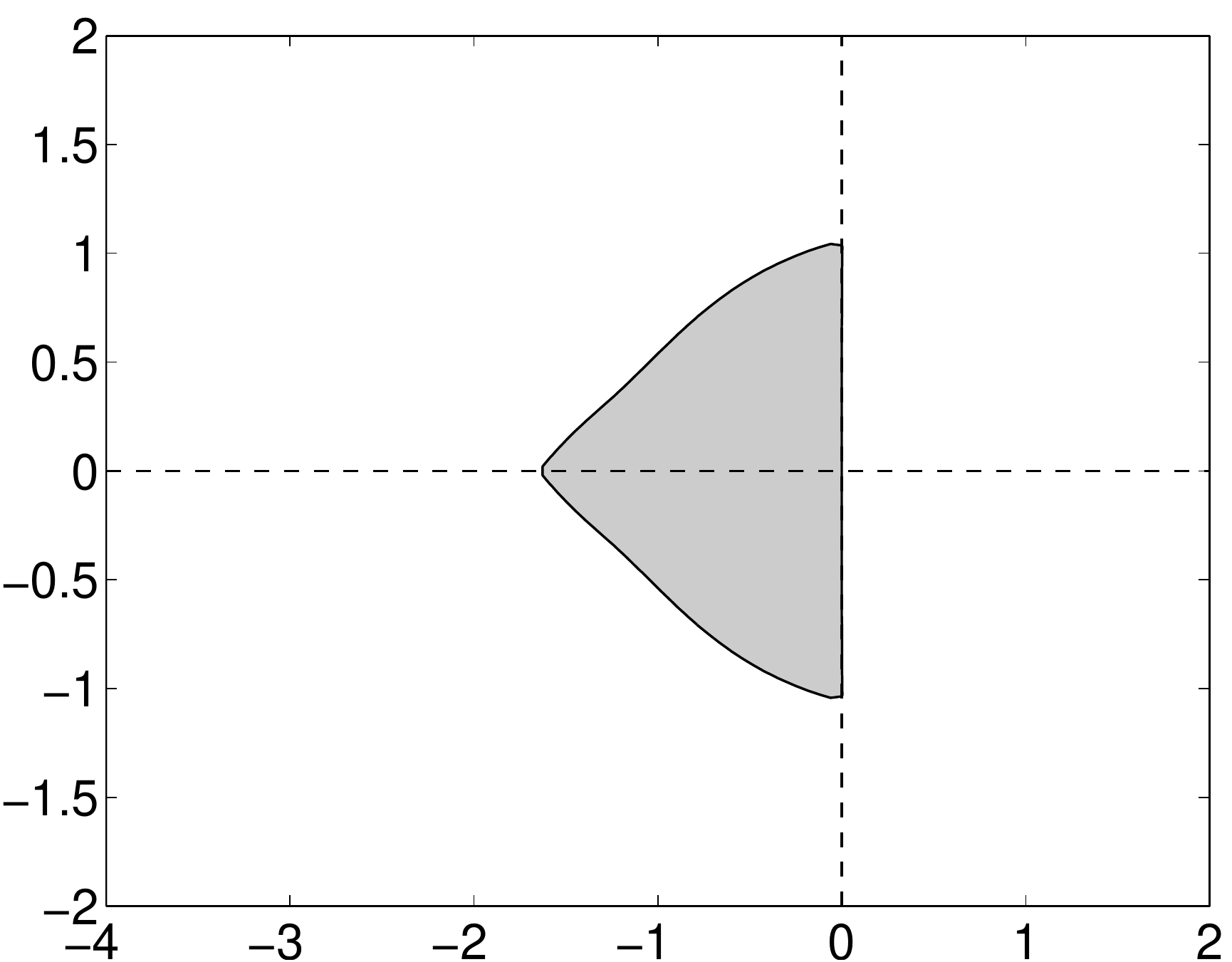}
 }
 \subfigure[Constrained stability region $\widehat{\mathcal{S}}_{\alpha}$ \eqref{eqn:stability-constrained} for
 $\alpha=90^\circ$]{
 \includegraphics[width=0.3\textwidth]{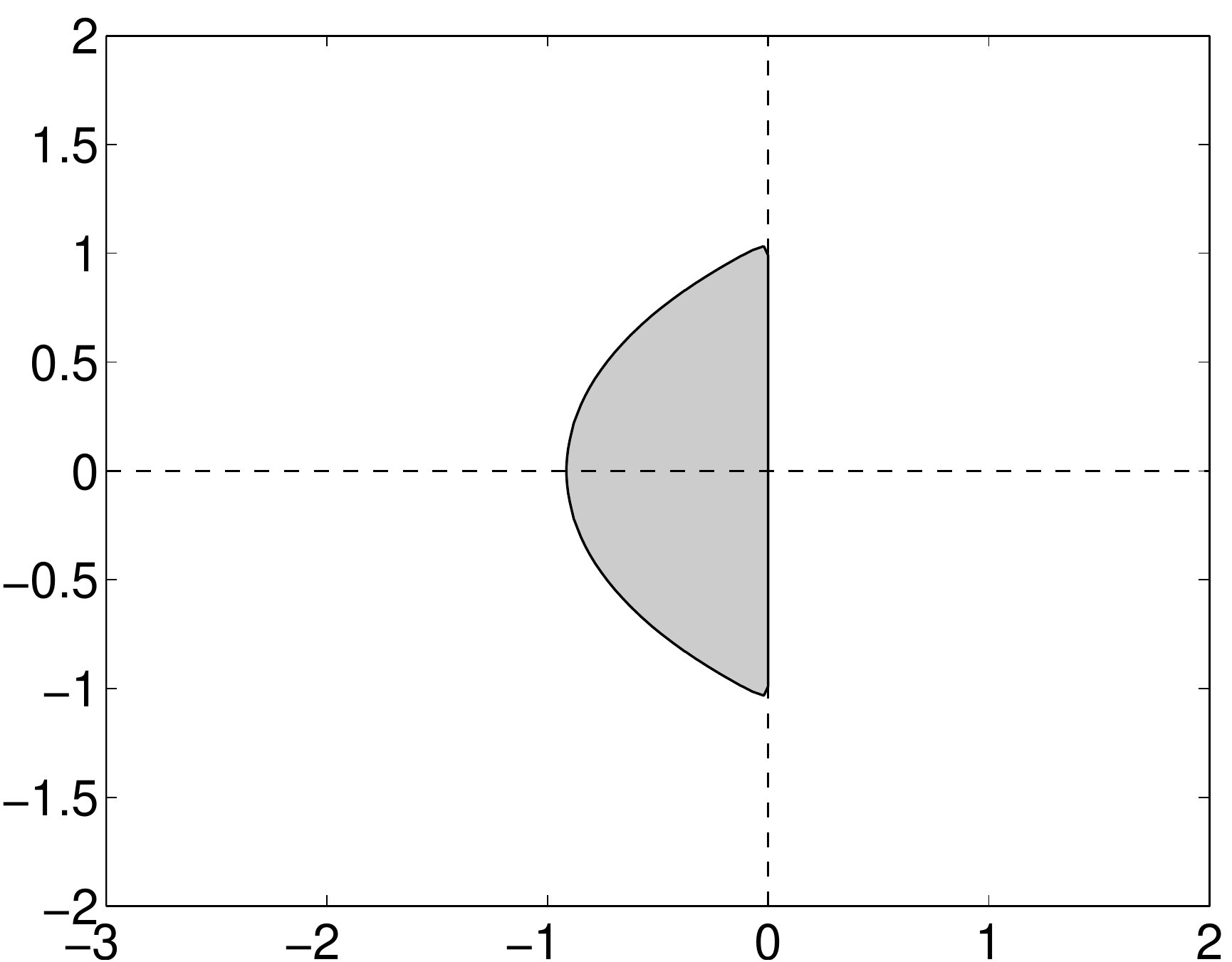}
 }
}
\caption{Stability regions for the IMEX-DIMSIM-3A pair of schemes\label{fig:sta_reg_ord3a}}
\end{figure}

\begin{figure} 
\centering{
\subfigure[Stability region $\widehat{S}$ of the implicit method]{
\includegraphics[width=0.3\textwidth]{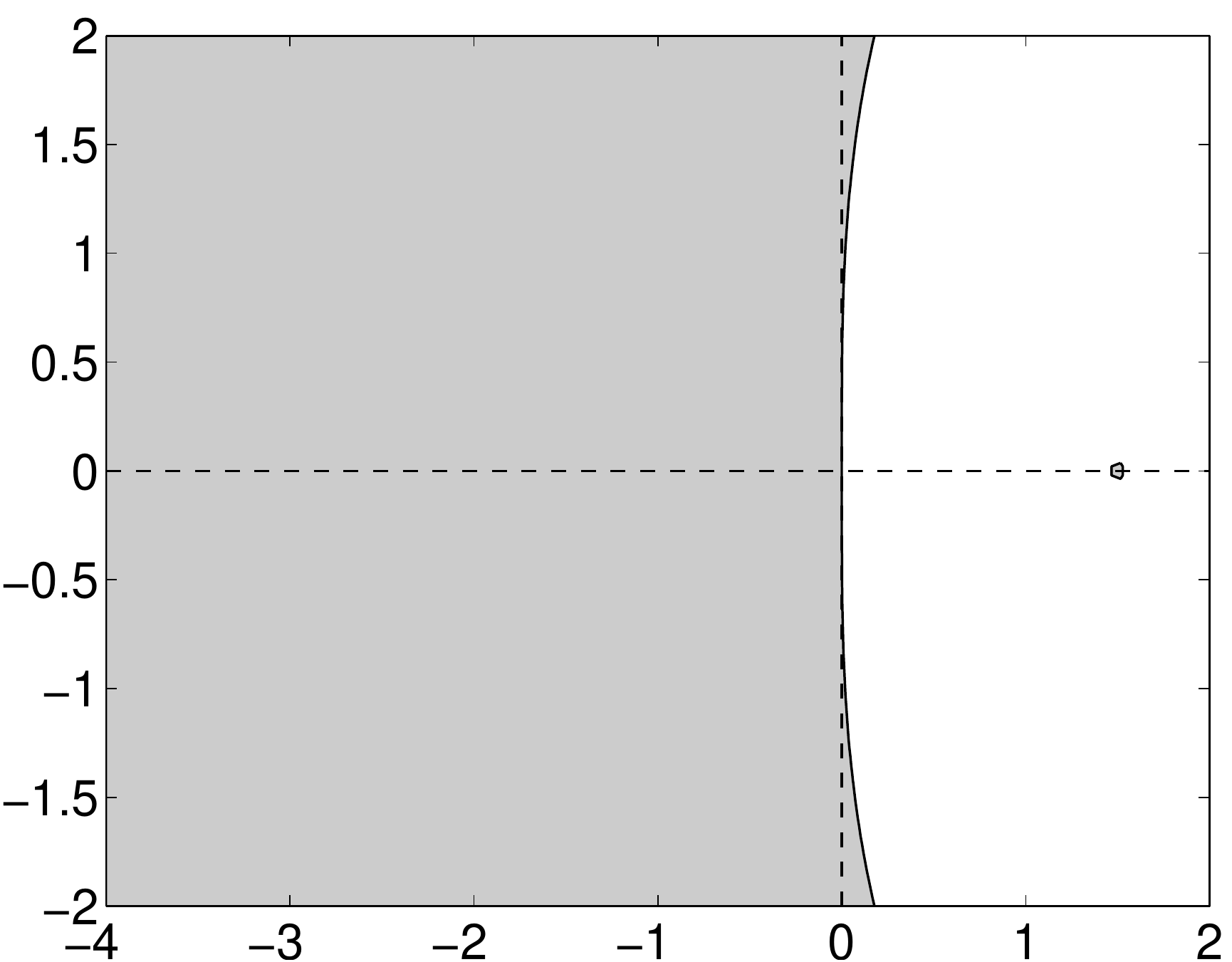}
}
 \subfigure[Stability region $S$ of the explicit method]{
 \includegraphics[width=0.3\textwidth]{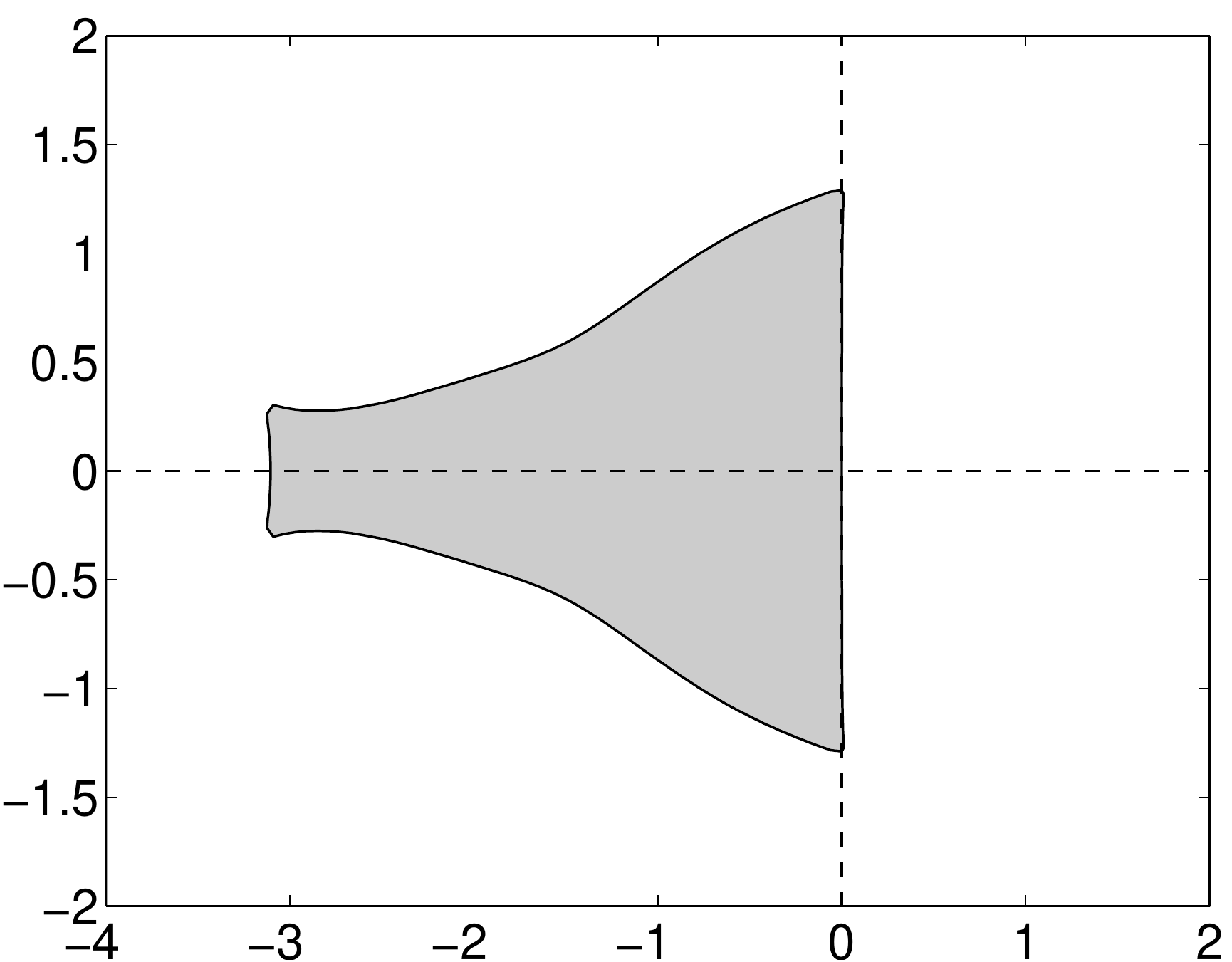}
 }
 \subfigure[Constrained stability region $\widehat{\mathcal{S}}_{\alpha}$ \eqref{eqn:stability-constrained} for
 $\alpha=90^\circ$]{
 \includegraphics[width=0.3\textwidth]{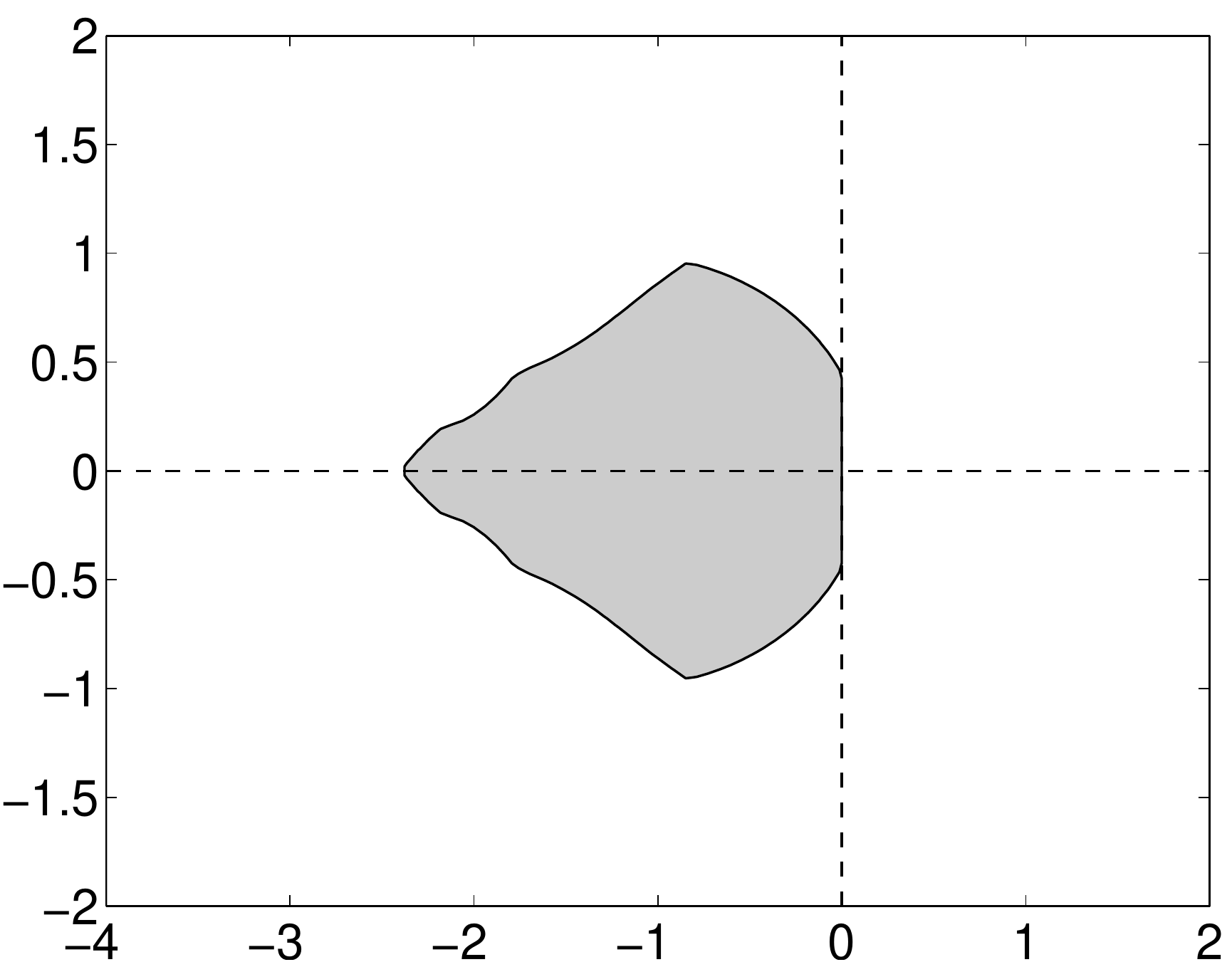}
 }
}
\caption{Stability regions for the IMEX-DIMSIM-3B pair of schemes\label{fig:sta_reg_ord3b}}
\end{figure}

\section{Numerical results}\label{sec:results}

We test the IMEX-GLM methods on two test problems. The first one is the van der Pol equation, a commonly used small ODE system that emphasizes convergence under stiffness.
The second test is a PDE problem arising in atmospheric modeling. 
We implemented our algorithms in a discontinuous Galerkin finite element model developed by Blaise et al. \cite{Blaise_2012}, 
which has efficient parallel scalability. 
We report the results obtained with IMEX-DIMSIM-2B and IMEX DIMSIM-3B methods, since they have the better accuracy and stability properties among their peers of the same order.

\subsection{Van der Pol equation}\label{sec:vdPol}

We consider the nonlinear van der Pol equation  with a split right hand side
\begin{equation}
\vect{y'\\z'} = f(y,z)+g(y,z)=\vect{z \\0} +\vect{0 \\ \left((1- y^2)z   -y\right)/\varepsilon}
\end{equation}
on the time interval $[0,0.5]$, with initial values
\begin{equation}
y(0) = 2, \quad z(0) = -\frac{2}{3}+\frac{1,0}{81}\varepsilon -\frac{292}{2187}\varepsilon^2- \frac{1814}{19683}\varepsilon^3+\mathcal{O}(\varepsilon^4).
\end{equation}
We consider  $\varepsilon = 10^{-6}$, a stiff case in which many methods suffer from order reduction \cite{Kennedy_2003}. 

The initialization \eqref{eqn:initialization-formula} was done using the analytic derivatives.
The reference solution is obtained with Radau-5, a stiffly accurate method \cite{Hairer_1993}, with very tight tolerances of 
$atol=rtol=5 \times {10^{-15}}$.
We compare the new methods with IMEX DIRK$(3,4,3)$, a L-stable, three-stage, third-order IMEX Runge-Kutta method proposed in \cite{Ascher_1997}. 

Figure \ref{fig:vdp} shows the global error, measured in the $L_2$ norm, against step size $h$. 
A geometric sequence of step sizes, $\tau$, $\tau/2$, $\tau/4$ and so on, were used. 
Order reduction can be clearly observed for the IMEX Runge-Kutta method, 
which yields second-order convergence. 
The IMEX DIMSIM converges at the theoretical third order and gives more accurate result than the IMEX Runge-Kutta method.        
Second-order IMEX DIMSIMs also produced no order reduction; detailed results have been reported in \cite{Zhang_2012}. 
These results indicate that the high stage order of IMEX DIMSIMs make them particularly attractive for solving stiff problems, where Runge-Kutta methods may suffer from order reduction.    

\begin{figure}
\centering{
 \includegraphics[width=0.50\textwidth]{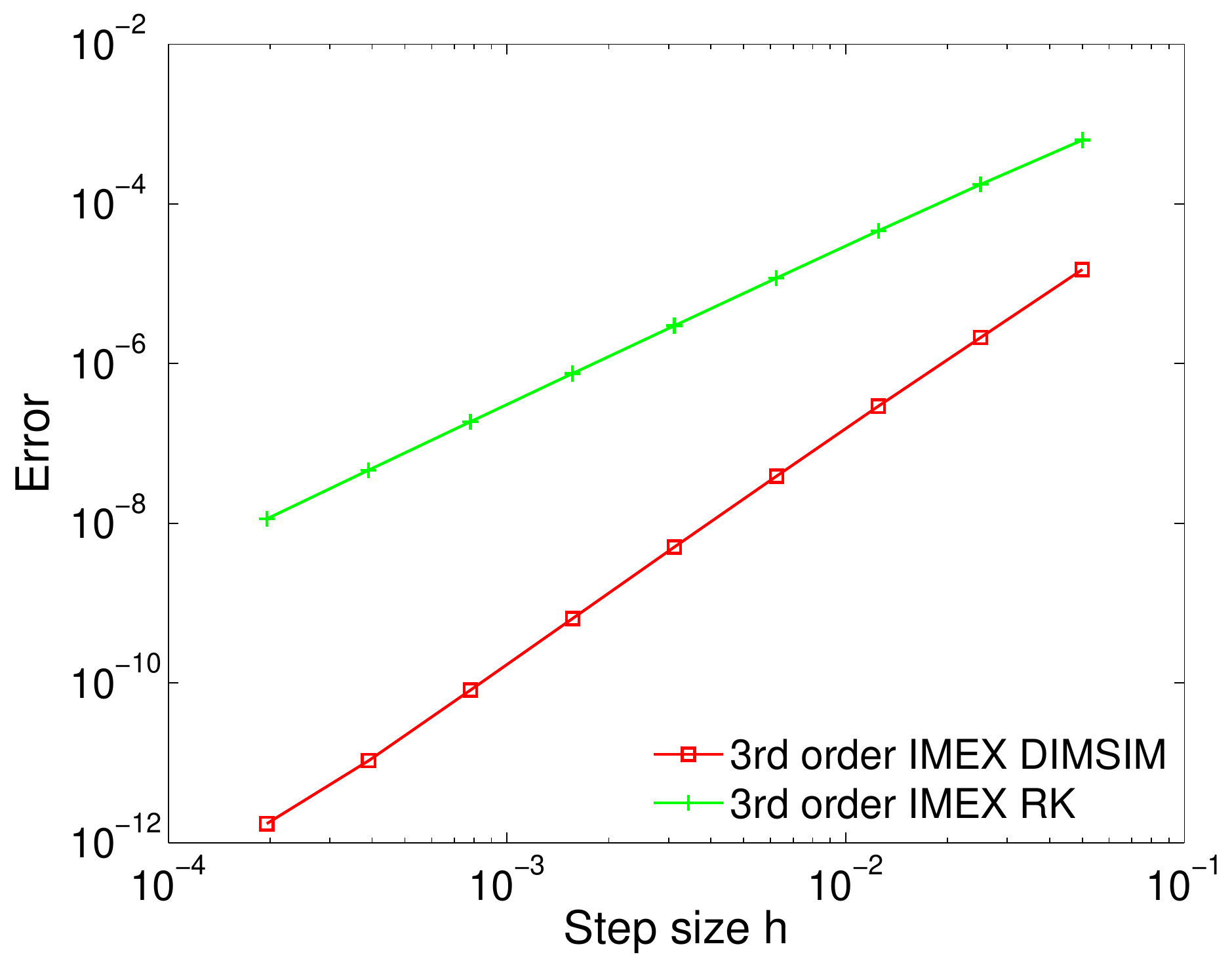}
\caption{Convergence results for third-order IMEX schemes on the van der Pol equation.}
\label{fig:vdp}
}
\end{figure}
%
\subsection{Gravity waves}\label{sec:gravity-waves}
Consider the dynamics of gravity waves, which is governed by the compressible Euler equation in the conservative form \cite{Giraldo_2010}
\begin{subequations}
 \label{eqn:gw}
\begin{eqnarray}
 \nonumber
\frac{\partial \rho}{\partial t} + \nabla \cdot \left( \rho \mathbf{u} \right) & = & 0 \\
 \label{eqn:gw-equation}
 \frac{\partial \rho \mathbf{u}}{\partial t} + \nabla \cdot \left( \rho \mathbf{u} \mathbf{u} + p \mathbf{I}\right) & = & - \rho g \mathbf{\widehat{e}_z} \\
 \nonumber
\frac{\partial \rho \theta}{\partial t} + \nabla \cdot \left( \rho \theta\mathbf{u} \right) & = & 0\,,
 \end{eqnarray}
where  $\rho$ is the density, 
$(\mathbf{u}$ is the velocity, $\theta$ is the potential temperature, and $\mathbf{I}$ is a $2 \times 2$ identity matrix. 
The prognostic variables are $(\rho, (\rho \mathbf{u})^T, (\rho \theta)^T)^T$.
The pressure $p$ in the momentum equation is computed by the equation of state
\begin{equation}
 \label{eqn:gw-pressure}
 p=p_0\left( \frac{\rho \theta R}{p_0} \right)^{\frac{c_p}{c_v}}.
\end{equation}
\end{subequations}
To maintain the hydrostatic state, we follow the splitting introduced in \cite{Giraldo_2010}
\begin{eqnarray*}
\rho(\mathbf{x},t) &=& \bar{\rho}(z) + \rho'(\mathbf{x},t) \\
(\rho \theta)(\mathbf{x},t)& =& \overline{(\rho \theta)} (z) + (\rho \theta)'(\mathbf{x},t) \\
p(\mathbf{x},t) &=& \bar{p}(z) + p'(\mathbf{x},t)
\end{eqnarray*}
where the reference (overlined) values are in hydrostatic balance. 
The gravity wave equation \eqref{eqn:gw} can be rewritten as
\begin{subequations}
\label{eqn:new-gravity-wave}
\begin{eqnarray}
\nonumber
 \frac{\partial \rho'}{\partial t} & = & - \nabla \cdot \left( \rho \mathbf{u} \right)  \\
\label{eqn:new-gravity-wave-eqn}
 \frac{\partial \rho \mathbf{u}}{\partial t} & = & - \nabla \cdot \left( \rho \mathbf{u} \mathbf{u} + p' \mathbf{I}\right) - \rho' g \mathbf{\widehat{e}_z} \\
\nonumber
 \frac{\partial (\rho \theta)'}{\partial t} & = & - \nabla \cdot \left( \rho \theta\mathbf{u} \right) \,,
\end{eqnarray}
closed by the equation of state
\begin{equation}
 p'=p_0\left( \frac{\rho \theta R}{p_0} \right)^{\frac{c_p}{c_v}}- \bar{p}. 
\end{equation}
\end{subequations}
The 2D mesh is generated by the software GMSH \cite{Gmsh}. 
The spatial discretization uses discontinuous Galerkin finite elements and was developed by Blaise et al. \cite{Blaise_2012}. 
Figure \ref{fig:snapshots} shows the density, velocity, potential temperature, and pressure variables after $900$ seconds of simulation time. 

The advantage of implicit-explicit time-stepping over explicit time-stepping schemes for this problem has been demonstrated in \cite{Seny_2013}. To apply IMEX integration the right-hand side of \eqref{eqn:new-gravity-wave-eqn} is additively split into linear and nonlinear parts. The linear term
\begin{equation}
 -
 \begin{bmatrix}
  \nabla \cdot \left( \rho \mathbf{u} \right) \\
  \nabla \cdot \left( p' \mathbf{I}\right)  + \rho' g \mathbf{\widehat{e}_z}\\
  \nabla \cdot \left( \rho \bar{\theta}\mathbf{u} \right)
 \end{bmatrix}
\end{equation}
with the pressure linearized as 
\begin{equation*}
 p'=\frac{\gamma \bar{p}}{\overline{\rho \theta}} \left( \rho \theta \right)'
\end{equation*}
is solved implicitly, while the remaining (nonlinear) terms are solved explicitly. 

All the experiments are performed on a workstation with $4$ Intel Xeon E5-2630 Processors (24 cores in total) using $12$ MPI threads.
Note that the parallelization is not implemented at time-stepping level but at the spatial discretization level, therefore the parallel performance does not be affect the comparison of various time integrators. 

Here we compare the performance of IMEX methods for a simulation window of
$30$ seconds. The second order methods are IMEX-DIMSIM-2B and L-stable, two-stage, second-order IMEX DIRK$(2,3,2)$ \cite{Ascher_1997}. The third order methods are IMEX-DIMSIM-3B and IMEX DIRK$(3,4,3)$ \cite{Ascher_1997}. The integrated $L_2$ errors for 
all prognostic variables are measured against a reference solution. 
The reference solution was obtained by applying an explicit RK method to solve the original (non-split) model with a very small time step $h=0.005$.

The error versus computational effort diagrams are shown in Figure \ref{fig:gw}.
All the methods display the theoretical orders of convergence. 
IMEX DIMSIMs and IMEX RK methods perform similarly, with IMEX DIMSIMs yielding slightly better accuracy when the same time steps are chosen. Also,
IMEX DIMSIMs are slightly more efficient in terms of CPU time than the IMEX RK methods of the same order. 
Note that this specific DG implementation requires the solution to be recovered at each time step, therefore the termination procedure has been applied after each each time step. The implementation can be optimized such as to apply the termination procedure only once at the end of the simulation; this would result in additional savings in computational cost. As the order increases, the number of stages required by an IMEX RK method grows rapidly due to order conditions, while an IMEX DIMSIM typically uses a number of stages equal to its order. Consequently, we expect that IMEX DIMSIM methods will become even more competitive for higher orders.  
\begin{figure}
\centering{
\subfigure[density]{
\includegraphics[width=0.45\textwidth]{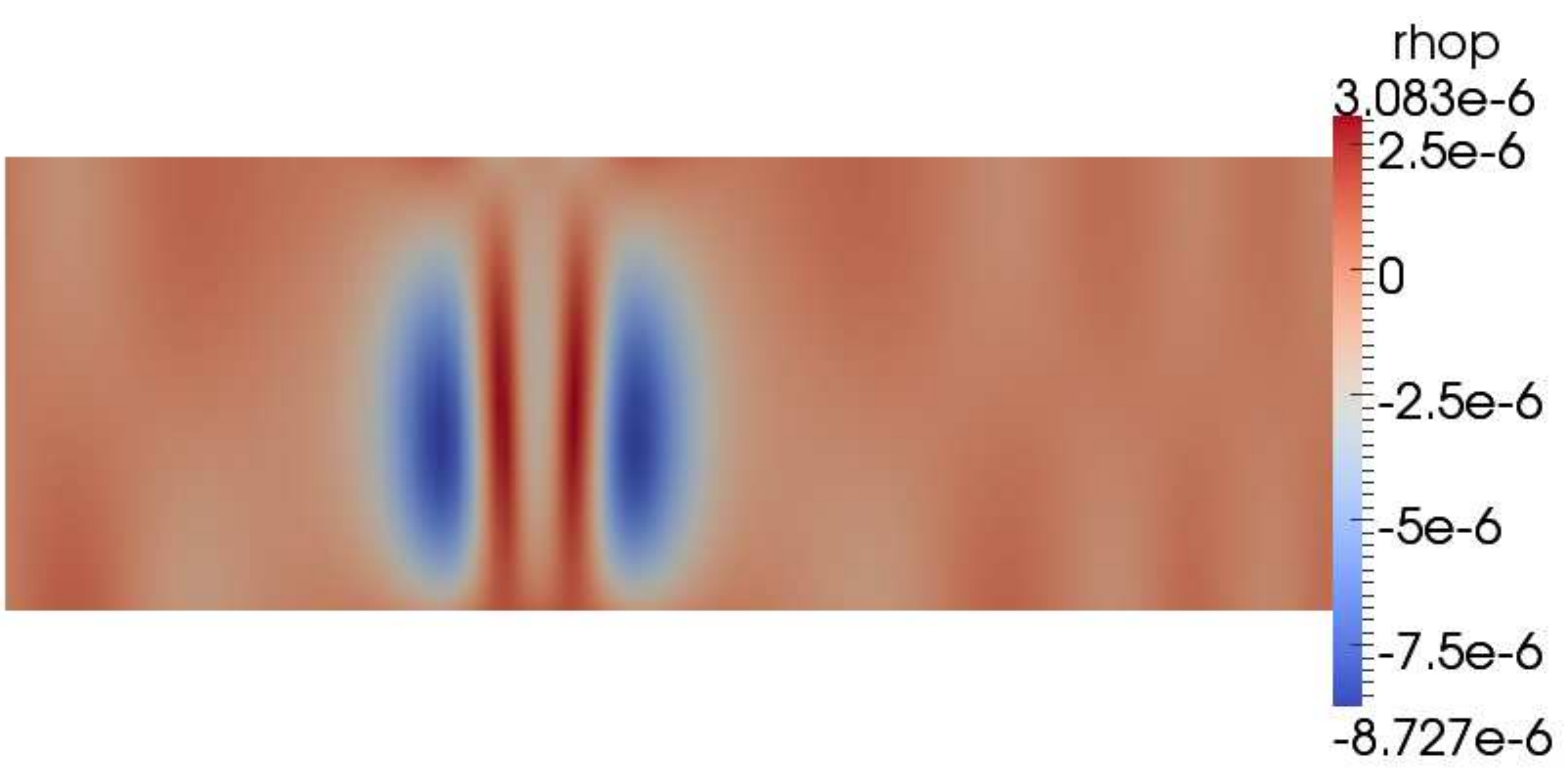}
}
 \subfigure[velocity]{
 \includegraphics[width=0.45\textwidth]{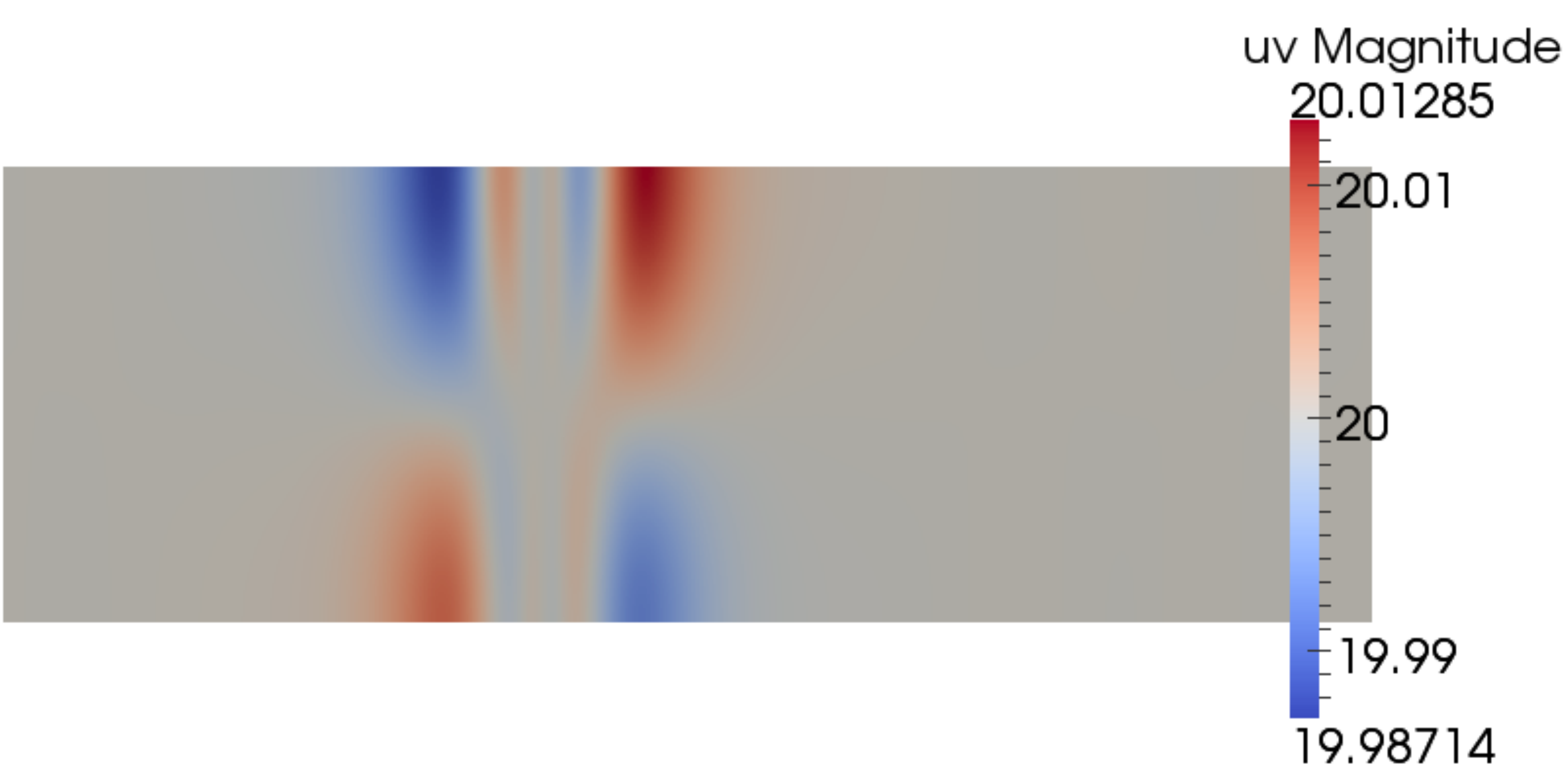}
 }
 \subfigure[potential temperature]{
 \includegraphics[width=0.45\textwidth]{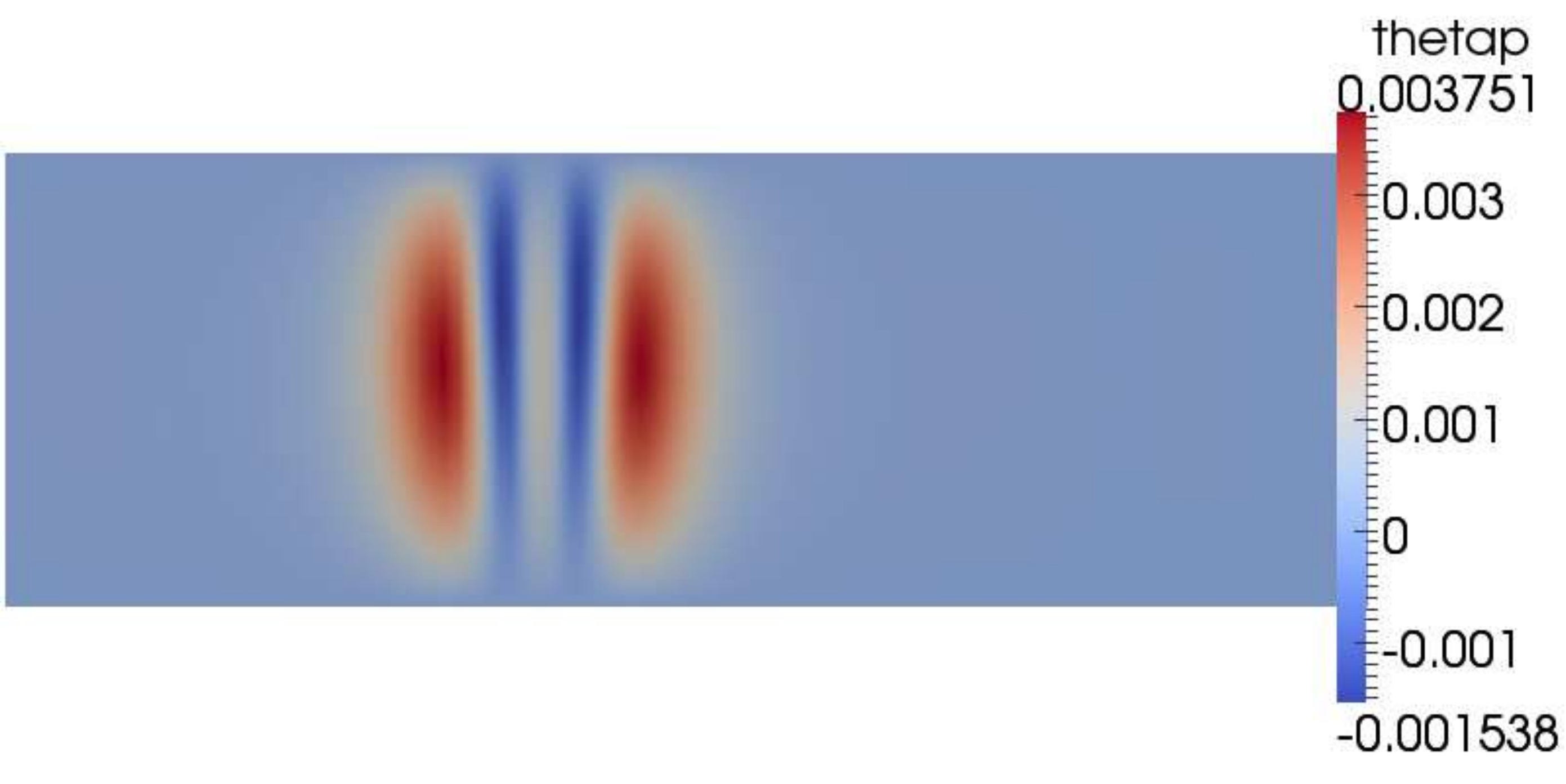}
 }
  \subfigure[pressure]{
 \includegraphics[width=0.45\textwidth]{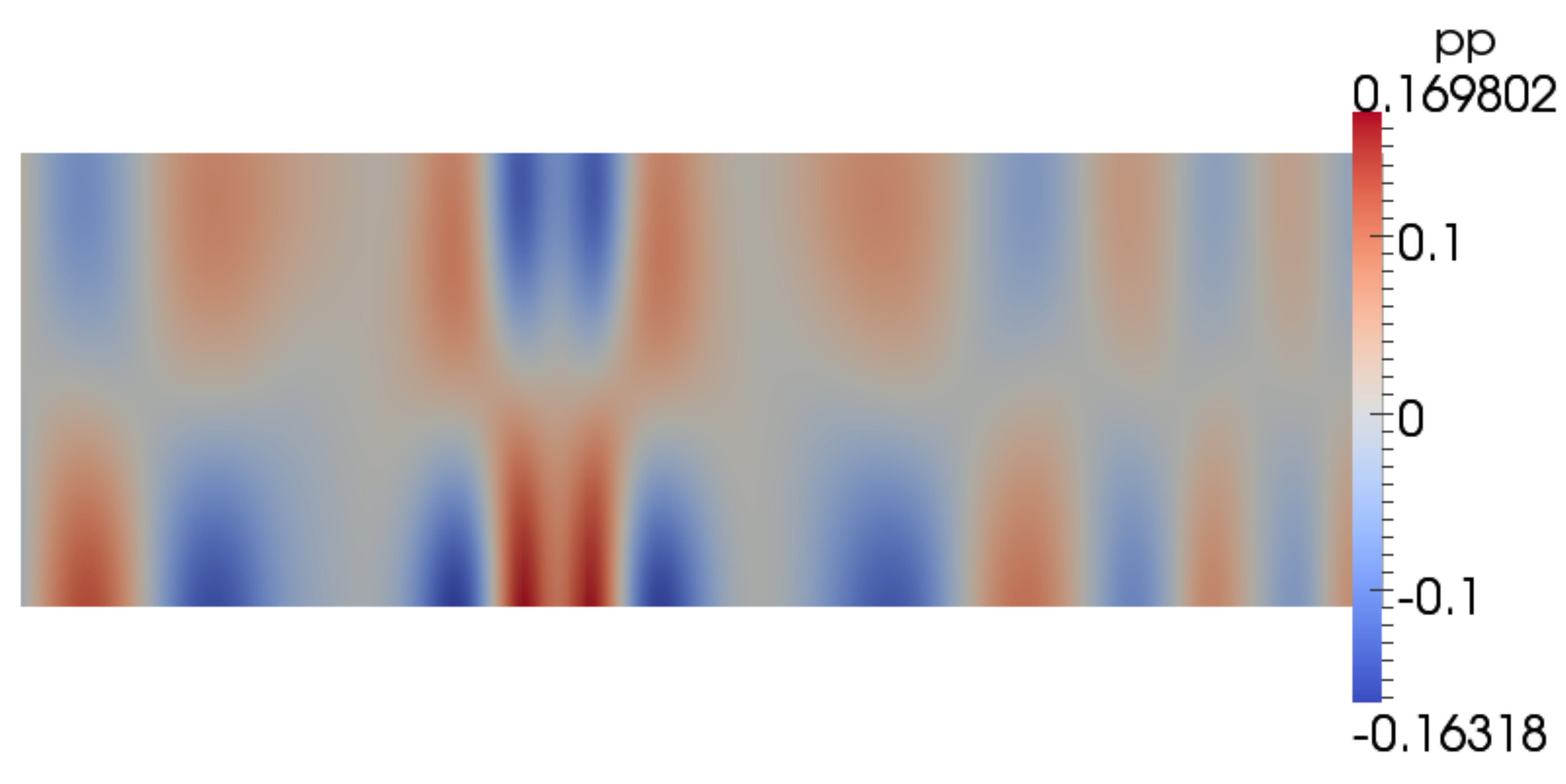}
 }
}
\caption{Solution of the gravity waves after $900$ simulation seconds. The results are obtained with a third-order discontinuous Galerking space discretization and third-order IMEX DIMSIM time integration.\label{fig:snapshots}}
\end{figure}
\begin{figure}
\centering{
  \subfigure[Convergence]{
  \includegraphics[width=0.45\textwidth]{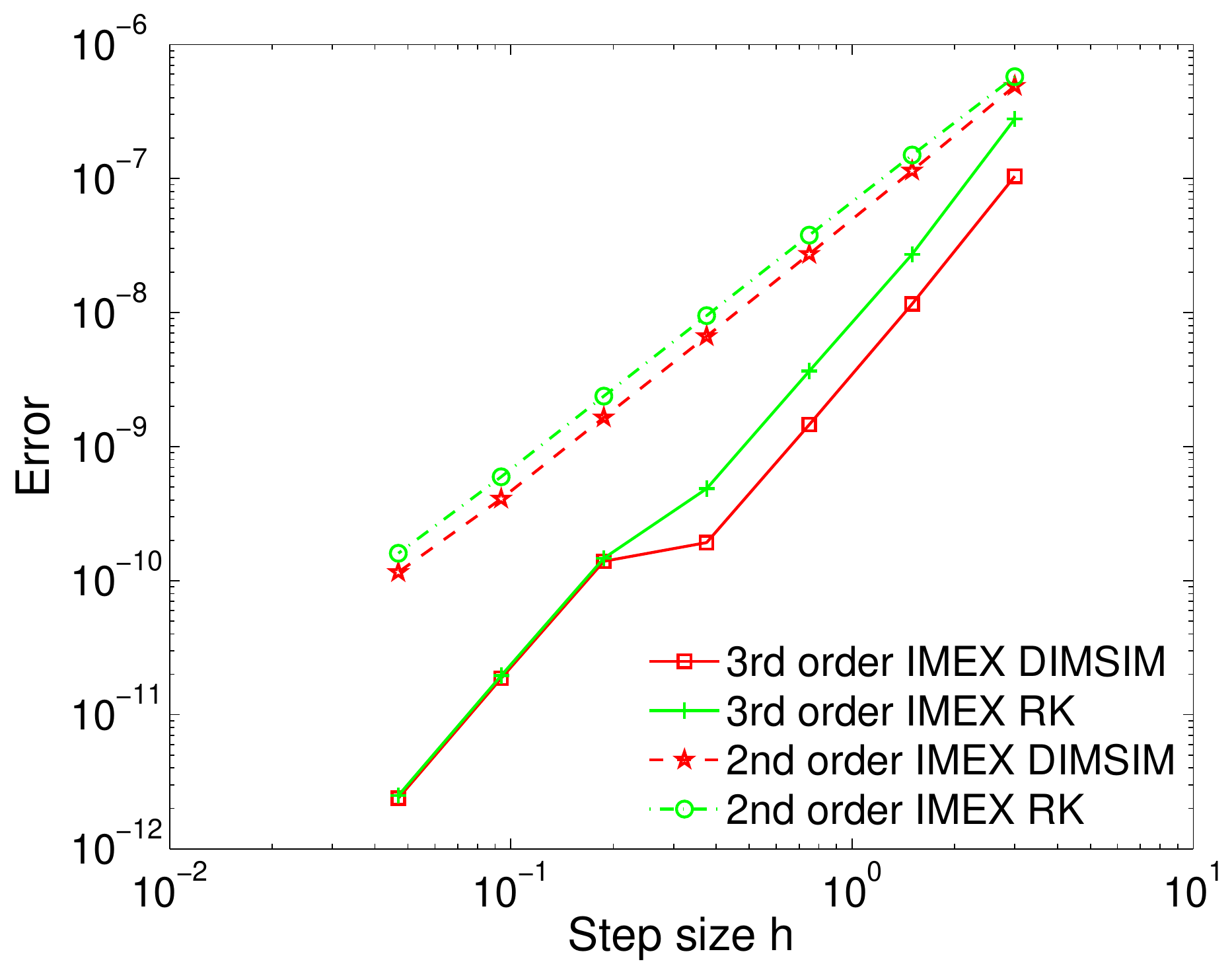}
  \label{fig:gw_convergence}
  }
  \subfigure[Work-precision diagram]{
  \includegraphics[width=0.45\textwidth]{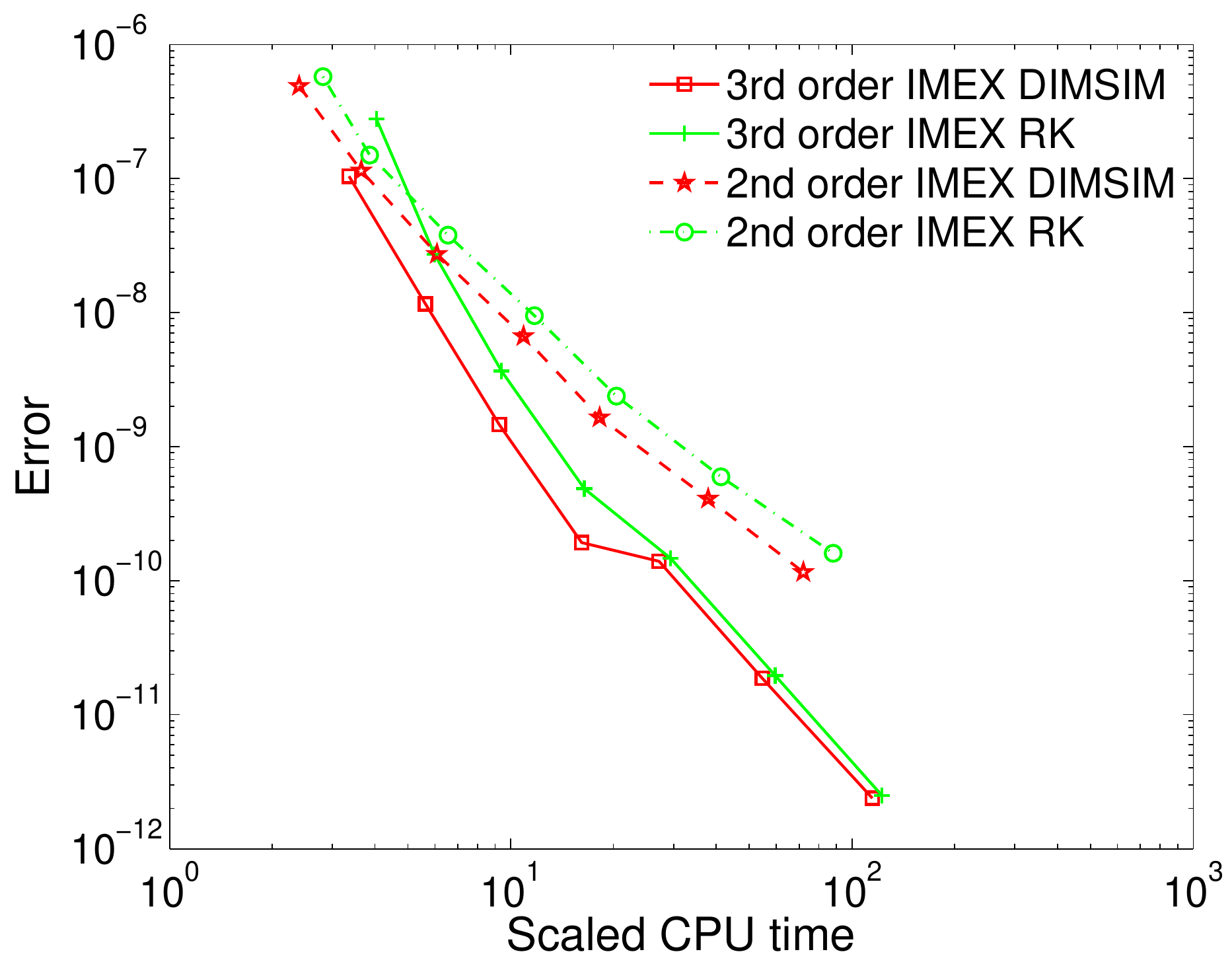}
  \label{fig:gw_work_precision}
  }
  \caption{\label{fig:gw}Integrated $L_2$ errors against time steps (a) and CPU time (b) for difference IMEX schemes. The errors are computed after $30$ s of simulation. 
  A geometric sequence of step sizes, $\tau$, $\tau/2$, $\tau/4$ and so on, is used.}
}
\end{figure}
%

\section{Conclusions and future work}\label{sec:conclusions}

In this paper introduce a new family of partitioned time integration methods based on high stage order general linear methods.
We prove that the general linear framework is well suited for the construction of multi-methods. 
Specifically, owing to the high stage orders,  no coupling conditions are needed to ensure the order of accuracy of the partitioned GLM.

We apply the partitioned general linear framework to construct new implicit-explicit GLM pairs, together with 
appropriate starting and ending procedures. The linear stability analysis proposes the use of constrained stability functions to quantify the joint stability of the IMEX pair. A Prothero-Robinson convergence analysis reveals that the order of an IMEX GLM scheme on very stiff problems is dictated by the stage order of its non-stiff component; in particular, no order reduction appears if the explicit method has a full stage order.
This result indicates that IMEX GLMs are particularly attractive for solving stiff problems, where other multistage methods may suffer from order reduction.    

We discuss the construction of practical IMEX GLM pairs starting from known implicit schemes and adding
an appropriate explicit counterpart. This strategy is applied to build second and third order IMEX diagonally-implicit-explicit multi-stage integration methods. 
Numerical experiments with the van der Pol equation confirm the fact that IMEX GLMs converge at full order while IMEX RK methods suffer from order reduction. The two dimensional gravity wave system is an important step towards solving real PDE-based problems. The new IMEX-DIMSIM schemes perform slightly better than the IMEX RK methods of the same order.

Future work will develop IMEX-GLMs of higher orders, will endow them with adaptive time stepping capabilities, 
and will study their advantages compared to other existing IMEX familiess. 
There are also implementation issues that deserve further exploration.

\section*{Acknowledgements}

The authors wish to thank Dr. Sebastien Blaise for making his GMSH/DG code, and the implementation of the  gravity waves problem, available for this work. We also thank him for his continuous support during our study.
This work has been supported in part by NSF through awards NSF
OCI-8670904397, NSF CCF-0916493, NSF DMS-0915047, NSF CMMI-1130667, NSF CCF -- 1218454
AFOSR FA9550--12--1--0293--DEF, FOSR 12-2640-06, DoD G\&C 23035,
and by the Computational Science Laboratory at Virginia Tech.


\end{document}